 \documentclass[final,3p,times]{elsarticle}
\usepackage{graphicx}
\usepackage{epstopdf}
\usepackage{amsmath,amssymb,enumerate,color,subfigure,amsthm}
\usepackage[table]{xcolor}
\biboptions{numbers,sort&compress}
\usepackage{comment,enumerate,multicol,xspace}
\usepackage{epsfig}
\usepackage{mathrsfs}
\usepackage{multirow,booktabs}
\usepackage[colorlinks,
            linkcolor=blue,
            anchorcolor=blue,
            citecolor=blue,
            ]{hyperref}

\usepackage{varioref}
\usepackage{aliascnt}
\newtheorem{thm}{Theorem}[section]
\newtheorem{exa}{Example}[section]
\newtheorem{lem}{Lemma}[section]

\labelformat{thm}{Theorem~#1}
\labelformat{lem}{Lemma~#1}
\labelformat{exa}{Example~#1}
\labelformat{figure}{Figure~#1}
\labelformat{table}{Table~#1}
\labelformat{section}{Section~#1}

\usepackage{float,graphicx}
\usepackage[font={bf,sf},textfont=md]{caption}
\restylefloat{table,figure}
\journal{International Journal of Computer Mathematics}

\begin{document}

\begin{frontmatter}



\title{A Fourth-Order Compact ADI Scheme for Two-Dimensional Riesz Space Fractional Nonlinear Reaction-Diffusion Equation}

\author{Dongdong Hu}
\author{Xuenian Cao\corref{cor1}}
\ead{cxn@xtu.edu.cn}

\cortext[cor1]{Corresponding author}
\address{School of Mathematics and Computational Science, Xiangtan University, Xiangtan 411105, PR China.}

\begin{abstract}
In this paper, a second-order backward difference formula (abbr. BDF2) is used to approximate first-order time partial derivative, the Riesz fractional derivatives are approximated by fourth-order compact operators, a class of new alternating-direction implicit difference scheme (abbr. ADI) is constructed for two-dimensional Riesz space fractional nonlinear reaction-diffusion equation. Stability and convergence of the numerical method are analyzed. Numerical experiments demonstrate that the proposed method is effective.
% Text of abstract

\end{abstract}

\begin{keyword}
Riesz space fractional derivative; BDF2 formula; Fourth-order compact operator; ADI scheme; Stability; Convergence



\end{keyword}

\end{frontmatter}

\section{Introduction}\label{1}
In this paper, we consider two-dimensional Riesz space fractional nonlinear reaction-diffusion equation \cite{space_time_one_order,space_time_one_two_order,spectralmethod,FitzHugh_Nagumo,Iyiola,Bueno-Orovio,Lin}
\begin{equation*}\label{eq1.1}
\begin{array}{ll}
\frac{\partial u(x,y,t)}{\partial t}={\kappa}_1 \frac{\partial^\alpha u(x,y,t)}{\partial |x|^\alpha}+{\kappa}_2\frac{\partial^\beta u(x,y,t)}{\partial |y|^\beta}+g(x,y,t,u(x,y,t)), \quad  (x,y,t) \in \Omega\times (0,T],
\end{array}\tag{1.1}
\end{equation*}
with the boundary and initial conditions
\begin{equation*}\label{eq1.2}
\begin{array}{ll}
u(x,y,t)=0,\quad (x,y,t)\in \partial{\Omega} \times (0,T],
\end{array}\tag{1.2}
\end{equation*}
\begin{equation*}\label{eq1.3}
\begin{array}{ll}
u(x,y,0)=\varphi(x,y), \quad (x,y)\in \overline{\Omega}=\partial{\Omega}\cup \Omega ,
\end{array}\tag{1.3}
\end{equation*}
where $1<\alpha,~\beta < 2$ and $\Omega=(a,b)\times(c,d)$, the diffusion coefficients $\kappa_1$, $\kappa_2$ are positive constants, $\varphi(x,y)$ is a known sufficiently smooth function, $g(x,y,t,u)$ satisfies the Lipschitz condition
\begin{equation*}\label{eq1.4}
|g(x,y,t,u)-g(x,y,t,\upsilon)|\leq L|u-\upsilon|,\forall u,\upsilon\in \mathbb{R},\tag{1.4}
\end{equation*}
here $L$ is Lipschitz constant, and Riesz fractional derivatives $\frac{\partial ^{\alpha}u(x,y,t)}{ \partial \left|  x  \right|^{\alpha} }$ and $\frac{\partial ^{\beta}u(x,y,t)}{ \partial \left|  y  \right|^{\beta} }$ are defined by
\begin{equation*}
\begin{array}{ll}
\frac{\partial ^{\alpha}u(x,y,t)}{ \partial \left|  x  \right|^{\alpha} }=c_{\alpha}\big(  {}_{a}^{R}\!D^{\alpha}_{x}+{}_{x}^{R}\!D^{\alpha}_{b}  \big)u(x,y,t)
,\quad
\frac{\partial ^{\beta}u(x,y,t)}{ \partial \left|  y  \right|^{\beta} }=c_{\beta}\big(  {}_{c}^{R}\!D^{\beta}_{y}+{}_{y}^{R}\!D^{\beta}_{d}  \big)u(x,y,t),
\end{array}
\end{equation*}
where $c_\gamma=-\frac{1}{2\cos(\frac{\pi\gamma}{2})}, \gamma=\alpha,~\beta$,
symbols ${}_{a}^{R}\!D^{\alpha}_{x}$,  ${}_{x}^{R}\!D^{\alpha}_{b}$, ${}_{c}^{R}\!D^{\beta}_{y}$ and ${}_{y}^{R}\!D^{\beta}_{d}$ denote left and right Riemann-Liouville fractional derivative operators, which are defined by \cite{define,define1,define2,define3}
\begin{equation*}\label{eq1.5}
{}_{a}^{R}\!D^{\alpha}_{x}u(x,y,t)=\frac{1}{\Gamma(2-\alpha)}\frac{\partial^2}{\partial x^2}\int_{a}^{x}u(\xi,y,t)(x-\xi)^{1-\alpha}d\xi,
\end{equation*}
\begin{equation*}\label{eq1.6}
{}_{x}^{R}\!D^{\alpha}_{b}u(x,y,t)=\frac{1}{\Gamma(2-\alpha)}\frac{\partial^2}{\partial x^2}\int_{x}^{b}u(\xi,y,t)(\xi-x)^{1-\alpha}d\xi,
\end{equation*}
\begin{equation*}\label{eq1.7}
{}_{c}^{R}\!D^{\beta}_{y}u(x,y,t)=\frac{1}{\Gamma(2-\beta)}\frac{\partial^2}{\partial y^2}\int_{c}^{y}u(x,\xi,t)(y-\xi)^{1-\beta}d\xi,
\end{equation*}
\begin{equation*}\label{eq1.8}
{}_{y}^{R}\!D^{\beta}_{d}u(x,y,t)=\frac{1}{\Gamma(2-\beta)}\frac{\partial^2}{\partial y^2}\int_{y}^{d}u(x,\xi,t)(\xi-y)^{1-\beta}d\xi,
\end{equation*}
where  ${\Gamma}$($\cdot$) is Gamma function.
\par
In this paper, we assumed that the problem \eqref{eq1.1}-\eqref{eq1.3} has a unique solution $u(x,y,t)\in C^{6,6,3}_{x,y,t}(~[a,b]\times[c,d]\times[0,T]~)$.  We also supposed for the fixed $t\in[0,T]$ and $y\in[c,d]$, $\widetilde{u}(x,\cdot,\cdot)\in \mathscr{C}^{4+\alpha}(\mathbb{R})$, for the fixed $t\in[0,T]$ and $x\in[a,b]$, $\widetilde{u}(\cdot,y,\cdot)\in \mathscr{C}^{4+\beta}(\mathbb{R})$, where $\widetilde{u}(x,\cdot,\cdot)$ and $\widetilde{u}(\cdot,y,\cdot)$ are defined as
\begin{equation*}
\widetilde{u}(x,\cdot,\cdot)=\left \{
\begin{array}{cl}
u(x,\cdot,\cdot),\qquad & x\in [a,b],\\
0,\qquad &  \mathbb{R} \backslash [a,b],
\end{array}
\right.
\qquad
\qquad
{
\widetilde{u}(\cdot,y,\cdot)=\left \{
\begin{array}{cl}
u(\cdot,y,\cdot),\qquad & y\in [c,d],\\
0,\qquad &  \mathbb{R} \backslash [c,d],
\end{array}
\right.
}
\end{equation*}
here $\mathscr{C}^{4+\alpha}(\mathbb{R})$ and $\mathscr{C}^{4+\beta}(\mathbb{R})$ are of the form
\begin{equation*}
\mathscr{C}^{4+\gamma}(\mathbb{R})=\Big\{ v| v\in L_1(\mathbb{R}) ,\int^{+\infty}_{-\infty}(1+|\varpi|)^{4+\gamma}|\widehat{v}(\varpi)|d\varpi<\infty \Big\},\qquad \gamma=\alpha,~\beta,
\end{equation*}
where $\widehat{v}(\varpi)$ is represented as the Fourier transformation of $v(x)$ and defined by
\begin{equation*}
\widehat{v}(\varpi)=\int^{+\infty}_{-\infty} e^{-i\varpi x}v(x)dx,\quad i^2=-1.
\end{equation*}

%\end{enumerate}
\par
In recent years, two-dimensional Riesz space fractional nonlinear reaction-diffusion equation plays an essential role in describing the propagation of the electrical potential in heterogeneous cardiac tissue \cite{space_time_one_order,spectralmethod,Bueno-Orovio,space_time_one_two_order,FitzHugh_Nagumo}, it attracts many author's attention in constructing numerical methods for problems of the form \eqref{eq1.1}-\eqref{eq1.3}. For approximation of Riesz derivative, Meerschaert and Tadjeran \cite{GL} initially proposed the shifted Gr\"unwald-Letnikov approximation with first-order accuracy for Riemann-Liouville fractional derivative. Based on this approximation, Tian et al. \cite{wsGL2} estabilished a second-order weighted and shifted Gr\"unwald-Letnikov approximation for Riemann-Liouville fractional derivative, and the approximation was applied in Riesz space fractional advection-dispersion equations \cite{wsGL1}. Hao et al. \cite{quasicompact} constructed a class of new weighted and shifted Gr\"unwald-Letnikov approximation with second-order accuracy, and it was applied in \cite{midpoint} for fractional Ginzburg-Landau equation. Ortigueira \cite{centeroperator2} initially proposed the fractional centered difference method with second-order accuracy for Riesz fractional derivative, and this method was applied in Riesz space fractional partial differential equation \cite{centeroperator,BDFoperator,space_time_one_two_order,Aiguo_Xiao_nonlinear_Schr?dinger,Furati,Yousuf,Yi. Tang}. Ding and Li \cite{new_generating_functions} proposed a novel second-order approximation for Riesz derivative via constructing a new generating function, and this second-order approximation was adopted in \cite{Hengfei} for two-dimension Riesz
space-fractional diffusion equation. Recently, compact difference operator has been focused on the fractional differential equations for increasing the spatial accuracy. Zhou et al. \cite{thirdquasicompact} constructed a third-order quasi-compact difference scheme for Riemann-Liouville fractional derivative. Hao et al. \cite{quasicompact} and Zhao et al. \cite{positiveoperator} proposed fourth-order compact difference operators to approximate Riemann-Liouville and Riesz derivatives, respectively, these compact difference operators have a great contribution on promoting algorithm accuracy. During these years, there also has developed some approximations by finite element method \cite{FitzHugh_Nagumo,Y.J Choi,Burrage}, spectral method \cite{spectralmethod,Bueno-Orovio,Lin} et al..
 As we noticed, for the approximation of first-order time derivative, implicit Euler method \cite{W.C.Hong,semi-implicit-difference,space_time_one_order,Y.J Choi}, Crank-Nicolson method \cite{Aiguo_Xiao_nonlinear_Schr?dinger,FitzHugh_Nagumo,spectralmethod,quasicompact,centeroperator,wsGL1,Hengfei}, implicit midpoint method \cite{midpoint,XNCaoXCaoLWen} and BDF2 method \cite{Liyunfei,BDFoperator,BDF_Volterra} are usually used, Pad\'e approximations which are based on Runge-Kutta method are also used in recent researches \cite{Yousuf,Furati}. And these methods have their own advantages for time-dependent problems.
\par
There are some researches \cite{space_time_one_order,space_time_one_two_order,spectralmethod,FitzHugh_Nagumo,Bueno-Orovio,Iyiola,Lin} on problem \eqref{eq1.1}-\eqref{eq1.3}. Liu et al. \cite{space_time_one_order,space_time_one_two_order} constructed two ADI finite difference schemes, where Riesz space derivatives were discretized by shifted Gr\"unwald-Letnikov formulae and fractional centered difference operators, respectively, implicit Euler method was applied to discretize time partial derivative, two proposed methods were proven to be stable and convergent. Bueno-Orovio et al. \cite{Bueno-Orovio} used Fourier spectral method to approximate Riesz space fractional derivative, implicit Euler method was adopted to discretize first-order time partial derivative, a semi-implicit Fourier spectral method was developed. Zeng et al. \cite{spectralmethod} and Bu et al. \cite{FitzHugh_Nagumo} applied Galerkin-Legendre spectral method and  Galerkin finite element method to approximate Riesz fractional derivative, respectively, two Crank-Nicolson ADI methods were established. Lin et al. \cite{Lin} used  a bivariate polynomial based
on shifted Gegenbauer polynomials method to approximate Riesz space fractional derivative, Runge-Kutta method of order 3 was applied to discretize the first-order time partial derivative, a Runge-Kutta Gegenbauer spectral method was constructed. Iyiola et al. also discussed several implicit-explicit schemes in \cite{Iyiola}. Because the computing scale of two dimensional diffusion equation problem is very big, so a more efficient algorithm is needed. So far, there are many high-order algorithms for Riesz fractional derivative, we noticed that the fourth-order fractional compact difference operator in \cite{positiveoperator} is symmetric positive definite under certain circumstance, it's helpful for us to analyze the stability and convergence. As we know, before ADI method is applied for solving two-dimensional nonlinear reaction-diffusion equation problem. the nonlinear source term needs to have linearized approximation. Therefore, how to deal with the nonlinear source term via linearized approximations \cite{space_time_one_order,space_time_one_two_order,Partially,Liyunfei} plays an important role in constructing ADI scheme for two-dimensional Riesz space fractional nonlinear reaction-diffusion equation. The objective of this paper is to try to use BDF method and the fourth-order fractional compact difference operator to construct a class of new high accuracy ADI scheme based on the first-order \cite{space_time_one_order} and second-order \cite{Partially} linearized approximations for nonlinear source term. Stability and convergence analysis are given by energy method.
\par
The outline of this paper is organized as follows. In \ref{sec2}, the numerical method is constructed for problem \eqref{eq1.1}-\eqref{eq1.3}. Then in \ref{sec3}, stability and convergence are discussed, respectively. In \ref{sec4}, we use the proposed method and these methods in literatures \cite{space_time_one_order,space_time_one_two_order} to solve the test problems. Numerical results show that the proposed method has high accuracy and efficiency.
\section{Numerical method}\label{sec2}
\par
Let $x_i=a+ih_x,~i=0,1,2,\cdots,M_1$, $y_j=c+jh_y,~j=0,1,2,\cdots,M_2$, $t_n=n\tau,~n=0,1,2,\cdots,N$, where $h_x=(b-a)/M_1$ and $h_y=(d-c)/M_2$ are spatial step sizes, $\tau=T/N$ denotes time step size. $u(x_i,y_j,t_n)$ and $u^n_{i,j}$ are exact solution and numerical solution of the problem \eqref{eq1.1}-\eqref{eq1.3} at $(x_i,y_j,t_n)$, respectively. We also denote $\overline{\Omega}_h=\{ (x_i,y_j)~|~0 \leq i \leq M_1,0\leq j \leq M_2 \}$, ~$\Omega_h=\overline{\Omega}_h\cap \Omega,$ and the boundary grid mesh is $\partial \Omega_h=\overline{\Omega}_h\cap \partial \Omega$.
\par
To discretize the Riesz space fractional derivative, we would introduce the centred difference operators which are defined by \cite{centeroperator2}
\begin{equation*}\label{eq2.1}
\Delta^{\alpha}_{x}u(x_i,y_j,t_n)=\frac{-1}{h_x^\alpha}\sum\limits^{i}_{k=i-M_1}g^{(\alpha)}_{k}u(x_i-kh_x,y_j,t_n),\tag{2.1}
\end{equation*}
and
\begin{equation*}\label{eq2.2}
\Delta^{\beta}_{y}u(x_i,y_j,t_n)=\frac{-1}{h_y^\beta}\sum\limits^{j}_{k=j-M_2}g^{(\beta)}_{k}u(x_i,y_j-kh_y,t_n),\tag{2.2}
\end{equation*}
where the coefficients $g^{(\gamma)}_{k}$ are determined by
\begin{equation*}\label{eq2.3}
~g^{(\gamma)}_0=\frac{\Gamma(\alpha+1)}{\Gamma^2(\alpha/2+1)},\quad g^{(\gamma)}_{k}=\Big( 1-\frac{\alpha+1}{\alpha/2+k}  \Big)g^{(\gamma)}_{k-1},\quad g^{(\gamma)}_{-k}=g^{(\gamma)}_{k}~,\quad\gamma=\alpha,~\beta,\quad k=1,2,\cdots,\tag{2.3}
\end{equation*}
then we have following lemma.
\begin{lem}\label{lem2.1}(see \cite{positiveoperator}.)
If $\widetilde{u}(x,\cdot,\cdot)\in \mathscr{C} ^{4+\alpha}(\mathbb{R})$, ~$\widetilde{u}(\cdot,y,\cdot)\in \mathscr{C} ^{4+\beta}(\mathbb{R})$, for the fixied step-sizes $h_x$ and $h_y$, it holds that
\begin{equation*}\label{eq2.4}
\begin{array}{ll}
\mathscr{B}^{\alpha}_x\frac{\partial^\alpha u(x_i,y_j,t_n)}{\partial |x|^\alpha}=\Delta^{\alpha}_xu(x_i,y_j,t_n)+O(h_x^4),\quad 1\leq i \leq M_1-1,1\leq j \leq M_2-1,0\leq n \leq N,\tag{2.4}
\end{array}
\end{equation*}
\begin{equation*}\label{eq2.5}
\begin{array}{cl}
\mathscr{B}^{\beta}_y\frac{\partial^\beta u(x_i,y_j,t_n)}{\partial |y|^\beta}=\Delta^{\beta}_yu(x_i,y_j,t_n)+O(h_y^4),\quad 1\leq i \leq M_1-1,1\leq j \leq M_2-1,0\leq n \leq N,\tag{2.5}
\end{array}
\end{equation*}
where the Fourth-order compact operators $\mathscr{B}^{\alpha}_x$ and $\mathscr{B}^{\beta}_y$ are defined as follows
\begin{equation*}
\mathscr{B}_x^{\alpha}u(x_i,y_j,t_n)=\left \{
\begin{array}{ll}
c^{\alpha}_2u(x_{i-1},y_j,t_n)+(1-2c^{\alpha}_2)u(x_i,y_j,t_n)+c^{\alpha}_2u(x_{i+1},y_j,t_n),\quad &1 \leq i \leq M_1-1,0\leq j\leq M_2,\\
u(x_{i},y_j,t_n),\quad & i=\{0,M_1\},0\leq j\leq M_2.
\end{array}
\right.
\end{equation*}
and
\begin{equation*}
\mathscr{B}_y^{\beta}u(x_i,y_j,t_n)=\left \{
\begin{array}{ll}
c^{\beta}_2u(x_i,y_{j-1},t_n)+(1-2c^{\beta}_2)u(x_i,y_j,t_n)+c^{\beta}_2u(x_i,y_{j+1},t_n),\quad &1 \leq j \leq M_2-1,0\leq i\leq M_1,\\
u(x_{i},y_j,t_n),\quad & j=\{0,M_2\},0\leq i\leq M_1,
\end{array}
\right.
\end{equation*}
where $c^\gamma_2=\frac{\gamma}{24} \in (\frac{1}{24},\frac{1}{12}),\gamma=\alpha,~\beta.$
\end{lem}
Before approximating the first-order partial derivative, we would introduce the properties of BDF operator.
\begin{lem}\label{lem2.2}(see \cite{BDFoperator}.)
For any positive integer $n$, if $u(\cdot,\cdot,t)\in C^3([0,T])$, then
\begin{equation*}\label{eq2.6}
\frac{\partial u(x_i,y_j,t_n)}{\partial t}=D_t^{(2)}u(x_i,y_j,t_n)+r^n_{i,j},\tag{2.6}
\end{equation*}
where
\begin{equation*}
D_t^{(2)}u(x_i,y_j,t_n)=\left \{
\begin{array}{ll}
\delta_tu(x_i,y_j,t_{\frac12}),& n=1,\\
\frac32\delta_tu(x_i,y_j,t_{n-\frac12})-\frac12\delta_tu(x_i,y_j,t_{n-\frac32}), & n\geq 2,
\end{array}
\right.
\end{equation*}
 $$\delta_tu(x_i,y_j,t_{n-\frac12})=\frac{1}{\tau}\Big(u(x_i,y_j,t_n)-u(x_i,y_j,t_{n-1})\Big),$$
and $r^n_{i,j}$ satisfies
\begin{equation*}\label{eq2.7}
\big|r^n_{i,j}\big| = \left \{
\begin{array}{ll}
O(\tau) , & n=1,\\
O(\tau^2), & n\geq 2.
\end{array}\tag{2.7}
\right.
\end{equation*}
\end{lem}
The nonlinear source term can be treated by following process
\begin{equation*}
g(x_i,y_j,t_n,u(x_i,y_j,t_n))= \left \{
\begin{array}{ll}
g\big(x_i,y_j,t_1,u(x_i,y_j,t_0)\big)+O(\tau),\quad & n=1,\\
g\big(x_i,y_j,t_n,2u(x_i,y_j,t_{n-1})-u(x_i,y_j,t_{n-2})\big)+O(\tau^2),\quad & n\geq 2.
\end{array}
\right.
\end{equation*}
Let
\begin{equation*}
\widehat{g}(x_i,y_j,t_n,u(x_i,y_j,t_n))= \left \{
\begin{array}{ll}
g\big(x_i,y_j,t_1,u(x_i,y_j,t_0)\big),\quad & n=1,\\
g\big(x_i,y_j,t_n,2u(x_i,y_j,t_{n-1})-u(x_i,y_j,t_{n-2})\big),\quad & n\geq 2.
\end{array}
\right.
\end{equation*}
At the point ($x_i,y_j,t_n$) \eqref{eq1.1} becomes
\begin{equation*}\label{eq2.8}
\begin{array}{ll}
\frac{\partial u(x_i,y_j,t_n)}{\partial t}=\kappa_{1}\frac{\partial^\alpha u(x_i,y_j,t_n)}{\partial |x|^\alpha}+\kappa_{2}\frac{\partial^\beta u(x_i,y_j,t_n)}{\partial |y|^\beta}+g(x_i,y_j,t_n,u(x_i,y_j,t_n)),\quad (x_i,y_j)\in {\Omega}_h ,1\leq n \leq N.
\end{array}\tag{2.8}
\end{equation*}
Multiplying by $\mathscr{B}^{\alpha}_x \mathscr{B}^{\beta}_y$ in \eqref{eq2.8}, we obtain from \autoref{lem2.1} that
\begin{equation*}\label{eq2.9}
\begin{array}{ll}
\mathscr{B}^{\alpha}_x \mathscr{B}^{\beta}_y\frac{\partial u(x_i,y_j,t_n)}{\partial t}=\mathscr{B}^{\beta}_y\delta^{\alpha}_xu(x_i,y_j,t_n)+\mathscr{B}^{\alpha}_x\delta^{\beta}_yu(x_i,y_j,t_n)+\mathscr{B}^{\alpha}_x \mathscr{B}^{\beta}_yg(x_i,y_j,t_n,u(x_i,y_j,t_n))+O(h_x^4+h_y^4),
\end{array}\tag{2.9}
\end{equation*}
where $\delta^{\alpha}_{x}=\kappa_1\Delta^{\alpha}_{x}$ and $\delta^{\alpha}_{y}=\kappa_2\Delta^{\alpha}_{y}$.
\par
Substituting  \eqref{eq2.6} into \eqref{eq2.9}, we obtain
\begin{equation*}\label{eq2.10}
\begin{array}{ll}
\mathscr{B}^{\alpha}_x \mathscr{B}^{\beta}_yD_t^{(2)}u(x_i,y_j,t_n)=\mathscr{B}^{\beta}_y\delta^{\alpha}_xu(x_i,y_j,t_n)+\mathscr{B}^{\alpha}_x\delta^{\beta}_yu(x_i,y_j,t_n)\\
\qquad\qquad\qquad\qquad\qquad+\mathscr{B}^{\alpha}_x \mathscr{B}^{\beta}_yg(x_i,y_j,t_n,u(x_i,y_j,t_n))-\mathscr{B}^{\alpha}_x \mathscr{B}^{\beta}_yr^n_{i,j}+O(h_x^4+h_y^4).
\end{array}\tag{2.10}
\end{equation*}
Adding a small error term $\tau^2\sigma_n^2\delta^{\alpha}_x\delta^{\beta}_yD_t^{(2)} u(x_i,y_j,t_n)$ on both side of \eqref{eq2.10}, we have
\begin{equation*}\label{eq2.11}
\begin{array}{ll}
\mathscr{B}^{\alpha}_x \mathscr{B}^{\beta}_yD_t^{(2)}u(x_i,y_j,t_n)+\tau^2\sigma_n^2\delta^{\alpha}_x\delta^{\beta}_yD_t^{(2)} u(x_i,y_j,t_n)=\mathscr{B}^{\beta}_y\delta^{\alpha}_xu(x_i,y_j,t_n)+\mathscr{B}^{\alpha}_x\delta^{\beta}_yu(x_i,y_j,t_n)\\
\qquad\qquad\qquad\qquad\qquad\qquad\qquad\qquad\qquad\qquad+\mathscr{B}^{\alpha}_x \mathscr{B}^{\beta}_y\widehat{g}(x_i,y_j,t_n,u(x_i,y_j,t_n))+\mathscr{R}^n_{i,j},
\end{array}\tag{2.11}
\end{equation*}
where $\sigma_1=1$ and $\sigma_n=\frac23,n\geq 2$. And there exists the positive constants $c_1$ and $c_2$ such that
\begin{equation*}\label{eq2.12}
\big|\mathscr{R}^n_{i,j}\big| \leq \left \{
\begin{array}{ll}
c_1(\tau+h_x^4+h_y^4) , & n=1,\\
c_2(\tau^2+h_x^4+h_y^4), & n\geq 2.
\end{array}\tag{2.12}
\right.
\end{equation*}
Omitting  the truncation error $\mathscr{R}^n_{i,j}$, we can obtain the numerical scheme for solving the problem \eqref{eq1.1}-\eqref{eq1.3} as follows
\begin{equation*}\label{eq2.13}
\begin{array}{ll}
\mathscr{B}^{\alpha}_x \mathscr{B}^{\beta}_yD_t^{(2)}u^{n}_{i,j}+\tau^2\sigma_n^2\delta^{\alpha}_x\delta^{\beta}_y D_t^{(2)}u^{n}_{i,j}=\mathscr{B}^{\beta}_y\delta^{\alpha}_xu^{n}_{i,j}+\mathscr{B}^{\alpha}_x\delta^{\beta}_yu^{n}_{i,j}+\mathscr{B}^{\alpha}_x \mathscr{B}^{\beta}_y g^n_{i,j},\quad (x_i,y_j)\in {\Omega}_h ,1\leq n \leq N,
\end{array}\tag{2.13}
\end{equation*}
where
\begin{equation*}
g^n_{i,j}= \left \{
\begin{array}{ll}
g\big(x_i,y_j,t_1,u^0_{i,j}\big),\quad & n=1,\\
g\big(x_i,y_j,t_n,2u^{n-1}_{i,j}-u^{n-2}_{i,j}\big),\quad & n\geq 2,
\end{array}
\right.
\end{equation*}
the boundary and initial conditions are
\begin{equation*}\label{eq2.14}
\begin{array}{ll}
u^n_{i,j}=0,\quad (x_i,y_j)\in \partial{\Omega}_h ,1\leq n \leq N,
\end{array}\tag{2.14}
\end{equation*}
\begin{equation*}\label{eq2.15}
\begin{array}{ll}
u^0_{i,j}=\varphi(x_i,y_j), \quad (x_i,y_j)\in \overline{\Omega}_h .
\end{array}\tag{2.15}
\end{equation*}
Multipling by $\tau\sigma_n$ in \eqref{eq2.13}, and factorizing it, we have
\begin{equation*}
\begin{array}{ll}\label{eq2.16}
(\mathscr{B}_x^{\alpha}-\tau\sigma_n\delta_x^{\alpha} )(\mathscr{B}_y^{\beta}-\tau\sigma_n \delta_y^{\beta} )u^{n}_{i,j}=\mathscr{H}u^{n-1}_{i,j}+\tau\sigma_n\mathscr{B}^{\alpha}_x \mathscr{B}^{\beta}_y g^n_{i,j},
\end{array}\tag{2.16}
\end{equation*}
where $\mathscr{H}u^{n-1}_{i,j}=(\mathscr{B}^{\alpha}_x \mathscr{B}^{\beta}_y+\tau^2\sigma^2_n\delta^{\alpha}_x\delta^{\beta}_y)(I-\sigma_n\tau D_t^{(2)})u^n_{i,j}$, and we can also rewrite it as follows
\begin{equation*}
\mathscr{H}u^{n-1}_{i,j}= \left \{
\begin{array}{ll}
\big(\mathscr{B}^{\alpha}_x \mathscr{B}^{\beta}_y+\tau^2\delta^{\alpha}_x\delta^{\beta}_y\big)u^0_{i,j},\quad & n=1,\\
\big(\mathscr{B}^{\alpha}_x \mathscr{B}^{\beta}_y+\frac49\tau^2\delta^{\alpha}_x\delta^{\beta}_y\big)\big(\frac43u^{n-1}_{i,j}-\frac13u^{n-2}_{i,j}\big),\quad & n\geq 2.
\end{array}
\right.
\end{equation*}
Introducing an intermediate variable $u^\ast_{i,j}$, let $u^\ast_{i,j}=(\mathscr{B}_y^{\beta}-\tau\sigma_n \delta_y^{\beta} )u^{n}_{i,j}$, therefore, we constructed a class of D'Yakonov ADI finite difference scheme for solving the problem \eqref{eq1.1}-\eqref{eq1.3} as follows

\begin{enumerate}[\textbf{Step} 1:]
\item
for the fixed $j\in\{ 1,2,\cdots,M_2-1 \}$, $\big\{ u^\ast_{i,j}\big| 1\leq i\leq M_1-1 \big\}$ can be calculated by
\begin{equation*}\label{eq2.17}
\begin{array}{ll}
(\mathscr{B}_x^{\alpha}-\tau\sigma_n\delta_x^{\alpha} )u^\ast_{i,j}=\mathscr{H}u^{n-1}_{i,j}+\tau\sigma_n\mathscr{B}^{\alpha}_x \mathscr{B}^{\beta}_y g^n_{i,j},\qquad 1\leq n \leq N,
\end{array}\tag{2.17}
\end{equation*}
with the boundary conditions
\begin{equation*}\label{eq2.18}
\begin{array}{rl}
u^\ast_{0,j}=(\mathscr{B}_y^{\beta}-\tau\sigma_n \delta_y^{\beta} )u^{n}_{0,j},\quad u^\ast_{M_1,j}=(\mathscr{B}_y^{\beta}-\tau\sigma_n \delta_y^{\beta} )u^{n}_{M_1,j},\qquad
1\leq j \leq M_2-1,\quad 1\leq n \leq N.
\end{array}\tag{2.18}
\end{equation*}
\item
for the fixed $i\in\{ 1,2,\cdots,M_1-1 \}$, $\big\{u^n_{i,j}\big| 1\leq j\leq M_2-1 \big\}$ can be obtained by
\begin{equation*}\label{eq2.19}
\begin{array}{ll}
(\mathscr{B}_y^{\beta}-\tau\sigma_n \delta_y^{\beta} )u^{n}_{i,j}=u^\ast_{i,j},\qquad 1\leq n \leq N,
\end{array}\tag{2.19}
\end{equation*}
the boundary and initial conditions are \eqref{eq2.14}-\eqref{eq2.15}.
\end{enumerate}
\section{Stability and convergence analysis}\label{sec3}
In order to analyze the stability and convergence of the method, we introduce some notations and lemmas.
\par
Let
$$\widehat{\gamma}_h=\Big\{ \zeta^n \Big|  \zeta^n=(\zeta_{0,0}^n,\cdots,\zeta_{M_1,0}^n,\cdots,\zeta_{0,M_2}^n,\cdots,\zeta_{M_1,M_2}^n ),~~\zeta^n_{i,j}=0~if~(x_i,y_j)\in \partial\Omega_h,0\leq n \leq N \Big\}.  $$
For any $u^n,v^n\in\widehat{\gamma}_h$, we define the following discrete inner product and corresponding norm
\begin{equation*}
\begin{array}{ll}
(u^n,v^n)=h_xh_y\sum\limits^{M_1-1}_{i=1}\sum\limits^{M_2-1}_{j=1}u_{i,j}^nv_{i,j}^n,\qquad \lVert u^n \rVert = \sqrt{(u^n,u^n)}.
\end{array}
\end{equation*}
\par
In addition, for any $u^n\in\widehat{\gamma}_h$, we denote $\big|u^n\big|=(|u_{0,0}^n|,\cdots,|u_{M_1,0}^n|,\cdots,|u_{0,M_2}^n|,\cdots,|u_{M_1,M_2}^n| )$, it implies $\big|u^n\big|\in\widehat{\gamma}_h $, and denote $u^n_{\ast,j}=(u^n_{0,j},\cdots,u^n_{M_1,j})$ and $u^n_{i,\ast}=(u^n_{i,0},\cdots,u^n_{i,M_2})$.
\\
\begin{lem}\label{lem3.1}(see \cite{BDFoperator}.)
For any positive integer n and real vector $v=(v^0,v^1,\cdots,v^n) \in \mathbb{R}^{n+1}$, we have
\begin{equation*}
\begin{array}{ll}
\frac{4\tau}{3}\sum\limits^n_{k=2}v^k(D_t^{(2)}v^k) \geq (v^n)^2-\frac13(v^{n-1})^2-(v^1)^2+\frac13(v^0)^2-\frac23(v^1-v^0)^2,& n \geq 2,\\
\frac{4\tau}{3}\sum\limits^n_{k=1}v^k(D_t^{(2)}v^k) \geq (v^n)^2-\frac13(v^{n-1})^2-\frac13(v^1)^2-\frac13(v^0)^2,& n \geq 1,\\
\frac{4\tau}{3}v^1(D_t^{(2)}v^1)=\frac23(v^{1})^2-\frac23(v^{0})^2+\frac23(v^1-v^0)^2,& n = 1.
\end{array}
\end{equation*}
\end{lem}
 It is easy to check that $\mathscr{B}^{\alpha}_x$ and $\mathscr{B}^{\beta}_y$ are symmetric positive definite and self-adjoint \cite{positiveoperator}, Following from Lemma 3.11 in \cite{positiveoperator}, one can prove that there exists the fractional symmetric positive definite difference operators $Q_x $ and $ Q_y$ such that $\mathscr{B}^{\alpha}_x=(Q_x)^2$ and $\mathscr{B}^{\beta}_y=(Q_y)^2$, here, $Q_x$ and $Q_y$ are also commutable.
\\
\begin{lem}\label{lem3.2}(see \cite{positiveoperator}.)
For any $u^n \in \widehat{\gamma}_h,$ it holds that
\begin{equation*}\label{eq*}
\begin{array}{ll}
\frac13 \lVert u^n \rVert^2 \leq \lVert u^n \rVert_{\mathscr{B}}^2 \leq \lVert u^n \rVert^2,
\end{array}
\end{equation*}
where $\lVert u^n \rVert_{\mathscr{B}}=\sqrt{ ( \mathscr{B}_x^{\alpha} \mathscr{B}_y^{\beta}u^n,u^n ) }=\sqrt{ ( Q_xQ_yu^n,Q_xQ_yu^n ) }$.
\end{lem}
\begin{lem}\label{lem3.3}(see \cite{BDFoperator,positiveoperator,centeroperator}.)
For any $u^n \in \widehat{\gamma}_h,$ it holds that
\begin{equation*}
\begin{array}{ll}
(\delta^{\alpha}_xu_{\ast,j}^n,u_{\ast,j}^n):=-(\Lambda_xu_{\ast,j}^n,\Lambda_xu_{\ast,j}^n) \leq 0,\quad & 1\leq j \leq M_2-1,\\
(\delta^{\beta}_yu_{i,\ast}^n,u_{i,\ast}^n):=-(\Lambda_yu_{i,\ast}^n,\Lambda_yu_{i,\ast}^n) \leq 0,\quad & 1\leq i \leq M_1-1,
\end{array}
\end{equation*}
where $\Lambda_x$, $\Lambda_y$ are represented as fractional symmetric positive definite difference operators such that $-\delta^\alpha_x=(\Lambda_x)^2$, $-\delta^\beta_y=(\Lambda_y)^2$.
\par
It is easy to verify that $$(\delta^{\alpha}_x\delta^{\beta}_yu^n,u^n) \geq 0.$$
\end{lem}
Theorem~4.1 in \cite{centeroperator} means that $-\delta^{\alpha}_x$ and $-\delta^{\beta}_y$  are symmetric positive definite operators. Therefore, with the help of commutativity of $\Lambda_x$ and $\Lambda_y$, we could introduce a semi-norm $\lVert u^n \rVert_{\delta}^2\triangleq(\delta^{\alpha}_x\delta^{\beta}_yu^n,u^n):=(\Lambda_x\Lambda_yu^n,\Lambda_x\Lambda_yu^n) $ be similar with \cite{positiveoperator,BDFoperator}.
\\
\begin{lem}\label{lem3.4}
For any $u^n \in \widehat{\gamma}_h,$ it holds that
\begin{equation*}
\begin{array}{ll}
(\mathscr{B}_y^{\beta}\delta^{\alpha}_xu^n,u^n) \leq 0,
\qquad(\mathscr{B}_x^{\alpha}\delta^{\beta}_yu^n,u^n) \leq 0.
\end{array}
\end{equation*}
\end{lem}
\begin{proof}
From \ref{lem3.3}, we have
\begin{equation*}
\begin{array}{ll}
(\mathscr{B}_y^{\beta}\delta^{\alpha}_xu^n,u^n)=(\delta^{\alpha}_xQ_yu^n,Q_yu^n)=h_xh_y\sum\limits^{M_1-1}_{i=1}\sum\limits^{M_2-1}_{j=1}(\delta^{\alpha}_xQ_yu^n_{i,j})(Q_yu^n_{i,j})\\
~\quad\qquad\qquad=h_y\sum\limits^{M_2-1}_{j=1}(\delta^{\alpha}_xQ_yu^n_{\ast,j},Q_yu^n_{\ast,j})\\
~\quad\qquad\qquad\leq 0.
\end{array}
\end{equation*}
Similarly, we can obtain
$$(\mathscr{B}_x^{\alpha}\delta^{\beta}_yu^n,u^n) \leq 0.$$
The proof is completed.
\end{proof}

\begin{lem}\label{lem3.5}
(Discrete Bellman Inequality) ~~Let~$\rho_1,~\rho_2 \geqslant 0, ~\tau>0,~\epsilon_0,~\epsilon_1,~\cdots,~\epsilon_{\widehat{N}}$ are a series of nonnegative real numbers, satisfying
\begin{equation*}
\epsilon_{n} \leq \rho_2+\rho_1  \tau\sum\limits^{n-1}_{k=0}\epsilon_k,\quad n=1,\cdots ,\widehat{N},
\end{equation*}
then it holds that
\begin{equation*}
\epsilon_{n} \leq \rho_2 e^{\rho_1 n \tau},\quad n=1,\cdots ,\widehat{N}.
\end{equation*}
\end{lem}
Assuming that $\widetilde{u}^n_{i,j}$ is the numerical solution for the numerical method \eqref{eq2.13}-\eqref{eq2.15} starting from another initial value. Denote $E^n=(E_{0,0}^n,\cdots,E_{M_1,0}^n,\cdots,E_{0,M_2}^n,\cdots,E_{M_1,M_2}^n )$, where $E^n_{i,j}=u^n_{i,j}-\widetilde{u}^n_{i,j}, ~(x_i,y_j)\in{\Omega}_h $, then we have following consequence.
\begin{thm}\label{thm3.1}
 For any positive real number~$\nu \in (0,1)$, if $0 < \tau \leq \tau_0=\frac{1-\nu}{9L}$, then the numerical scheme \eqref{eq2.13}-\eqref{eq2.15} is stable, i.e.
 $$
\lVert E^n \rVert \leq \frac{3}{\nu}e^{\frac{18}{\nu}LT}\lVert \big( I+\frac{23}{18}\tau^2\delta^{\alpha}_x\delta^{\beta}_y \big)E^0 \rVert,\qquad n\geq 1.
$$
 \end{thm}
\begin{proof}
According to \eqref{eq2.13}-\eqref{eq2.15}, we have the following equations
\begin{equation*}\label{eq3.1}
\left \{
\begin{array}{ll}
\mathscr{B}^{\alpha}_x \mathscr{B}^{\beta}_yD_t^{(2)}E^{n}_{i,j}+\tau^2\sigma_n^2\delta^{\alpha}_x\delta^{\beta}_y D_t^{(2)}E^{n}_{i,j}=\mathscr{B}^{\beta}_y\delta^{\alpha}_xE^{n}_{i,j}+\mathscr{B}^{\alpha}_x\delta^{\beta}_yE^{n}_{i,j}+\mathscr{B}^{\alpha}_x \mathscr{B}^{\beta}_y g^n_{i,j}-\mathscr{B}^{\alpha}_x \mathscr{B}^{\beta}_y \widetilde{g}^n_{i,j},\\
\qquad\qquad\qquad\qquad\qquad\qquad\qquad\qquad\qquad\qquad\qquad(x_i,y_j)\in {\Omega}_h ,1\leq n \leq N,\\
E^n_{i,j}=0,\quad (x_i,y_j)\in \partial{\Omega}_h ,1\leq n \leq N,\\
E^0_{i,j}=\varphi_0(x_i,y_j), \quad (x_i,y_j)\in \overline{\Omega}_h .
\end{array}
\right.
\eqno{
\begin{array}{ll}
(3.1a)\\\\ (3.1b)\\(3.1c)
\end{array}
}
\end{equation*}
Multiplying by $\tau h_xh_yE^n_{i,j}$ in $(3.1a)$, then summing from 1 to $M_1-1$ on $i$, and summing from 1 to $M_2-1$ on $j$, we obtain
\begin{equation*}\label{eq3.2}
\begin{array}{ll}
\tau(\mathscr{B}^{\alpha}_x \mathscr{B}^{\beta}_yD_t^{(2)}E^{n},E^n)+\tau(\tau^2\sigma_n^2\delta^{\alpha}_x\delta^{\beta}_y D_t^{(2)}E^{n},E^n)=\tau(\mathscr{B}^{\beta}_y\delta^{\alpha}_xE^{n},E^n)+\tau(\mathscr{B}^{\alpha}_x\delta^{\beta}_yE^{n},E^n)\\
\qquad\qquad\qquad\qquad\qquad\qquad\qquad\qquad\qquad+\tau h_xh_y\sum\limits^{M_1-1}_{i=1}\sum\limits^{M_2-1}_{j=1}\mathscr{B}^{\alpha}_x \mathscr{B}^{\beta}_y \big( g^n_{i,j}-\widetilde{g}^n_{i,j} \big)E^n_{i,j}.
\end{array}\tag{3.2}
\end{equation*}
where
\begin{equation*}
\widetilde{g}^n_{i,j}= \left \{
\begin{array}{ll}
g\big(x_i,y_j,t_1,\widetilde{u}^0_{i,j}\big),\quad & n=1,\\
g\big(x_i,y_j,t_n,2\widetilde{u}^{n-1}_{i,j}-\widetilde{u}^{n-2}_{i,j}\big),\quad & n\geq 2.
\end{array}
\right.
\end{equation*}
From \ref{lem3.4}, we have
\begin{equation*}\label{eq3.3}
\begin{array}{ll}
\tau(\mathscr{B}^{\beta}_y\delta^{\alpha}_xE^{n},E^n)+\tau(\mathscr{B}^{\alpha}_x\delta^{\beta}_yE^{n},E^n) \leq 0.
\end{array}\tag{3.3}
\end{equation*}
Substituting \eqref{eq3.3} into \eqref{eq3.2}, it holds that
\begin{equation*}\label{eq3.4}
\begin{array}{ll}
\tau(\mathscr{B}^{\alpha}_x \mathscr{B}^{\beta}_yD_t^{(2)}E^{n},E^n)+\tau(\tau^2\sigma_n^2\delta^{\alpha}_x\delta^{\beta}_y D_t^{(2)}E^{n},E^n)
\leq \tau h_xh_y\sum\limits^{M_1-1}_{i=1}\sum\limits^{M_2-1}_{j=1}\mathscr{B}^{\alpha}_x \mathscr{B}^{\beta}_y \big( g^n_{i,j}-\widetilde{g}^n_{i,j} \big)E^n_{i,j}.
\end{array}\tag{3.4}
\end{equation*}
Multiplying by $\frac43$ in \eqref{eq3.4}, and summing from 1 to $k$ on $n$ and replacing k by n, we get
\begin{equation*}\label{eq3.5}
\begin{array}{ll}
\frac{4\tau}{3}\sum\limits^{n}_{k=1}(\mathscr{B}^{\alpha}_x \mathscr{B}^{\beta}_yD_t^{(2)}E^{k},E^k)+\frac{4\tau}{3}\sum\limits^{n}_{k=1}(\tau^2\sigma_k^2\delta^{\alpha}_x\delta^{\beta}_y D_t^{(2)}E^{k},E^k)\\
\qquad\qquad\qquad\qquad\qquad\qquad\leq \frac{4\tau h_xh_y}{3}\sum\limits^{n}_{k=1}\sum\limits^{M_1-1}_{i=1}\sum\limits^{M_2-1}_{j=1}\mathscr{B}^{\alpha}_x \mathscr{B}^{\beta}_y \big( g^k_{i,j}-\widetilde{g}^k_{i,j} \big)E^k_{i,j}.
\end{array}\tag{3.5}
\end{equation*}
According to \ref{lem3.1}, we have
\begin{equation*}\label{eq3.6}
\begin{array}{ll}
\frac{4\tau}{3}\sum\limits^{n}_{k=1}(\mathscr{B}^{\alpha}_x \mathscr{B}^{\beta}_yD_t^{(2)}E^{k},E^k)=\frac{4\tau}{3}\sum\limits^{n}_{k=1}(D_t^{(2)}Q_xQ_yE^{k},Q_xQ_yE^k)\\
~~~\quad\qquad\qquad\qquad\qquad\geq
\lVert E^n \rVert_\mathscr{B}^2-\frac13 \lVert E^{n-1} \rVert_\mathscr{B}^2-\frac13 \lVert E^1 \rVert_\mathscr{B}^2-\frac13 \lVert E^0 \rVert_\mathscr{B}^2,
\end{array}\tag{3.6}
\end{equation*}
and
\begin{equation*}\label{eq3.7}
\begin{array}{ll}
\frac{4\tau}{3}\sum\limits^{n}_{k=1} \sigma_k^2(\delta^{\alpha}_x\delta^{\beta}_y D_t^{(2)}E^{k},E^k)&=\frac{4\tau}{3}\sum\limits^{n}_{k=1} \sigma_k^2(D_t^{(2)}\Lambda_x\Lambda_y E^{k},\Lambda_x\Lambda_yE^k)\\
\qquad\qquad\qquad\qquad\qquad&=\frac49\cdot\frac{4\tau}{3}\sum\limits^{n}_{k=2} (D_t^{(2)}\Lambda_x\Lambda_y E^{k},\Lambda_x\Lambda_yE^k)\\
\qquad\qquad\qquad\qquad\qquad&\quad+\frac{4\tau}{3}(D_t^{(2)}\Lambda_x\Lambda_y E^{1},\Lambda_x\Lambda_yE^1)\\
\qquad\qquad\qquad\qquad\qquad&\geq
\frac49\lVert E^n \rVert_{\delta}^2-\frac{4}{27}\lVert E^{n-1} \rVert_{\delta}^2-\frac49\lVert E^1 \rVert_{\delta}^2+\frac{4}{27}\lVert E^0 \rVert_{\delta}^2-\frac{8}{27}\lVert E^1-E^0 \rVert_{\delta}^2\\
\qquad\qquad\qquad\qquad\qquad&\quad+\frac23\lVert E^1 \rVert_{\delta}^2-\frac23\lVert E^0 \rVert_{\delta}^2+\frac23\lVert E^1-E^0 \rVert_{\delta}^2\\
\qquad\qquad\qquad\qquad\qquad&\geq
\frac49\lVert E^n \rVert_{\delta}^2-\frac{4}{27}\lVert E^{n-1} \rVert_{\delta}^2-\frac{14}{27}\lVert E^0 \rVert_{\delta}^2.
\end{array}\tag{3.7}
\end{equation*}
Substituting \eqref{eq3.6}-\eqref{eq3.7} into \eqref{eq3.5}, we obtain
\begin{equation*}\label{eq3.8}
\begin{array}{ll}
\lVert E^n \rVert^2_\mathscr{B}+\frac49\tau^2\lVert E^n \rVert_{\delta}^2
\leq
\frac13\big( \lVert E^{n-1} \rVert^2_\mathscr{B}+\frac49\tau^2\lVert E^{n-1} \rVert_{\delta}^2 \big)+\frac13\lVert E^1 \rVert^2_\mathscr{B}+\frac13\big( \lVert E^0 \rVert^2_\mathscr{B}+\frac{14}{9}\tau^2\lVert E^0 \rVert_{\delta}^2 \big)\\
\quad\qquad\qquad\qquad\qquad+\frac{4\tau h_xh_y}{3} \sum\limits^{n}_{k=1}\sum\limits^{M_1-1}_{i=1}\sum\limits^{M_2-1}_{j=1}\big|\mathscr{B}^{\alpha}_x \mathscr{B}^{\beta}_y \big( g^k_{i,j}-\widetilde{g}^k_{i,j} \big)E^k_{i,j}\big|.
\end{array}\tag{3.8}
\end{equation*}
When $k=1$, from \ref{lem3.2}, \eqref{eq1.4} and Cauchy-Schwarz inequality, we have
\begin{equation*}\label{eq3.9}
\begin{array}{ll}
h_xh_y\sum\limits^{M_1-1}_{i=1}\sum\limits^{M_2-1}_{j=1}\big|\mathscr{B}^{\alpha}_x \mathscr{B}^{\beta}_y \big( g^1_{i,j}-\widetilde{g}^1_{i,j} \big)E^1_{i,j}\big|\leq(\mathscr{B}^{\alpha}_x \mathscr{B}^{\beta}_y\big| g^1-\widetilde{g}^1 \big | ,|E^1|)\\
\quad\qquad\qquad\qquad\qquad\qquad\qquad\leq
L(\mathscr{B}^{\alpha}_x \mathscr{B}^{\beta}_y\big| E^{0} \big | ,\big | E^1 \big | )\\
\quad\qquad\qquad\qquad\qquad\qquad\qquad=
L(Q_xQ_y\big| E^{0} \big | ,Q_xQ_y\big | E^1 \big | )\\
\quad\qquad\qquad\qquad\qquad\qquad\qquad\leq
\frac{L}{2}\big\lVert |E^0| \big\rVert^2_{\mathscr{B}}+\frac{L}{2}\big\lVert |E^1| \big\rVert^2_{\mathscr{B}}\\
\quad\qquad\qquad\qquad\qquad\qquad\qquad\leq
\frac{L}{2}\big\lVert |E^0| \big\rVert^2+\frac{L}{2}\big\lVert |E^1| \big\rVert^2\\
\quad\qquad\qquad\qquad\qquad\qquad\qquad=
\frac{L}{2}\big\lVert E^0 \big\rVert^2+\frac{L}{2}\big\lVert E^1 \big\rVert^2.
\end{array}\tag{3.9}
\end{equation*}
When $k\geq 2$, according to \ref{lem3.2}, \eqref{eq1.4} and Cauchy-Schwarz inequality, we get
\begin{equation*}\label{eq3.10}
\begin{array}{ll}
h_xh_y\sum\limits^{M_1-1}_{i=1}\sum\limits^{M_2-1}_{j=1}\big|\mathscr{B}^{\alpha}_x \mathscr{B}^{\beta}_y \big( g^k_{i,j}-\widetilde{g}^k_{i,j} \big)E^k_{i,j}\big|\leq(\mathscr{B}^{\alpha}_x \mathscr{B}^{\beta}_y\big| g^k-\widetilde{g}^k \big | ,|E^k|)\\
\qquad\qquad\qquad\qquad\qquad\leq
L(\mathscr{B}^{\alpha}_x \mathscr{B}^{\beta}_y\big| 2E^{k-1}-E^{k-2} \big | ,|E^k|)\\
\qquad\qquad\qquad\qquad\qquad\leq
2L(\mathscr{B}^{\alpha}_x \mathscr{B}^{\beta}_y\big| E^{k-1}\big | ,\big | E^k \big |)+L(\mathscr{B}^{\alpha}_x \mathscr{B}^{\beta}_y\big| E^{k-2}\big | ,\big | E^k \big |)\\
\qquad\qquad\qquad\qquad\qquad\leq
2L\big\lVert |E^{k-1}| \big\rVert_{\mathscr{B}}\big\lVert |E^k| \big\rVert_{\mathscr{B}}+L\big\lVert |E^{k-2}| \big\rVert_{\mathscr{B}}\big\lVert |E^k| \big\rVert_{\mathscr{B}}\\
\qquad\qquad\qquad\qquad\qquad\leq
L\big\lVert |E^{k-1}| \big\rVert^2_{\mathscr{B}}+L\big\lVert |E^k| \big\rVert^2_{\mathscr{B}}+\frac{L}{2}\big\lVert |E^{k-2}| \big\rVert_{\mathscr{B}}^2+\frac{L}{2}\big\lVert |E^k| \big\rVert_{\mathscr{B}}^2\\
\qquad\qquad\qquad\qquad\qquad=
\frac{3L}{2}\big\lVert |E^k| \big\rVert^2_{\mathscr{B}}+L\big\lVert |E^{k-1}| \big\rVert^2_{\mathscr{B}}+\frac{L}{2}\big\lVert |E^{k-2}| \big\rVert^2_{\mathscr{B}}\\
\qquad\qquad\qquad\qquad\qquad\leq
\frac{3L}{2}\big\lVert |E^k| \big\rVert^2+L\big\lVert |E^{k-1}| \big\rVert^2+\frac{L}{2}\big\lVert |E^{k-2}| \big\rVert^2\\
\qquad\qquad\qquad\qquad\qquad=
\frac{3L}{2}\big\lVert E^k \big\rVert^2+L\big\lVert E^{k-1} \big\rVert^2+\frac{L}{2}\big\lVert E^{k-2} \big\rVert^2.
\end{array}\tag{3.10}
\end{equation*}
Substituting \eqref{eq3.9}-\eqref{eq3.10} into \eqref{eq3.8}, from \ref{lem3.2}, we have
\begin{equation*}\label{eq3.11}
\begin{array}{ll}
\lVert E^n \rVert^2_{\mathscr{B}} \leq \lVert E^n \rVert^2_{\mathscr{B}}+\frac49\tau^2\lVert E^n \rVert_{\delta}^2\\
\quad\qquad\leq
\frac13\big( \lVert E^{n-1} \rVert^2_{\mathscr{B}}+\frac49\tau^2\lVert E^{n-1} \rVert_{\delta}^2 \big)+\frac13\lVert E^1 \rVert^2_{\mathscr{B}}+\frac13\big( \lVert E^0 \rVert^2_{\mathscr{B}}+\frac{14}{9}\tau^2\lVert E^0 \rVert_{\delta}^2 \big)\\
\qquad\qquad+\frac{4\tau L}{3}\sum\limits^{n}_{k=2}\big(  \frac32\lVert E^k \rVert^2+\lVert E^{k-1} \rVert^2+\frac12\lVert E^{k-2} \rVert^2  \big)\\
\qquad\qquad+
\frac{2\tau L}{3}\lVert E^1 \rVert^2+\frac{2\tau L}{3}\lVert E^0 \rVert^2\\
\quad\qquad\leq
\frac13\big( \lVert E^{n-1} \rVert^2_{\mathscr{B}}+\frac49\tau^2\lVert E^{n-1} \rVert_{\delta}^2 \big)+\frac13\lVert E^1 \rVert^2_{\mathscr{B}}+\frac13\big( \lVert E^0 \rVert^2_{\mathscr{B}}+\frac{14}{9}\tau^2\lVert E^0 \rVert_{\delta}^2 \big)\\
\qquad\qquad+4\tau L\sum\limits^{n}_{k=2}\big(  \frac32\lVert E^k \rVert^2_{\mathscr{B}}+\lVert E^{k-1} \rVert^2_{\mathscr{B}}+\frac12\lVert E^{k-2} \rVert^2_{\mathscr{B}}  \big)\\
\qquad\qquad+
2\tau L\lVert E^1 \rVert^2_{\mathscr{B}}+2\tau L\lVert E^0 \rVert^2_{\mathscr{B}}.
\end{array}\tag{3.11}
\end{equation*}
Taking $n=1$ in \eqref{eq3.11}, we find
\begin{equation*}\label{eq3.12}
\begin{array}{ll}
\lVert E^1 \rVert^2_{\mathscr{B}} \leq \lVert E^0 \rVert^2_{\mathscr{B}}+\tau^2\lVert E^0 \rVert^2_{\delta}+3\tau L\lVert E^1 \rVert^2_{\mathscr{B}}+3\tau L\lVert E^0 \rVert^2_{\mathscr{B}}.
\end{array}\tag{3.12}
\end{equation*}
Substituting \eqref{eq3.12} into \eqref{eq3.11}, we obtain
\begin{equation*}\label{eq3.13}
\begin{array}{ll}
\lVert E^n \rVert^2_{\mathscr{B}}+\frac49\tau^2\lVert E^n \rVert_{\delta}^2
\leq
\frac13\big( \lVert E^{n-1} \rVert^2_{\mathscr{B}}+\frac49\tau^2\lVert E^{n-1} \rVert_{\delta}^2 \big)+\frac23\big( \lVert E^0 \rVert^2_{\mathscr{B}}+\frac{23}{18}\tau^2\lVert E^0 \rVert_{\delta}^2 \big)\\
\quad\qquad\qquad\quad\qquad\qquad+3\tau L\lVert E^1 \rVert^2_{\mathscr{B}}+3\tau L\lVert E^0 \rVert^2_{\mathscr{B}}\\
\quad\qquad\qquad\quad\qquad\qquad+4\tau L\sum\limits^{n}_{k=2}\big(  \frac32\lVert E^k \rVert^2_{\mathscr{B}}+\lVert E^{k-1} \rVert^2_{\mathscr{B}}+\frac12\lVert E^{k-2} \rVert^2_{\mathscr{B}}  \big).
\end{array}\tag{3.13}
\end{equation*}
Taking $0 \leq n_0 \leq n$ such that
\begin{equation*}\label{eq3.14}
\begin{array}{ll}
\lVert E^{n_0} \rVert^2_{\mathscr{B}}+\frac49\tau^2\lVert E^{n_0} \rVert_{\delta}^2= \underset{0 \leq l \leq n}{\max}\big(
\lVert E^{l} \rVert^2_{\mathscr{B}}+\frac49\tau^2\lVert E^{l} \rVert_{\delta}^2
 \big) \geq \lVert E^{l} \rVert^2_{\mathscr{B}},\quad 0 \leq l \leq n.
 \end{array}\tag{3.14}
\end{equation*}
Therefore
\begin{equation*}
\begin{array}{ll}
\frac13\big( \lVert E^{n-1} \rVert^2_{\mathscr{B}}+\frac49\tau^2\lVert E^{n-1} \rVert_{\delta}^2 \big) \leq
\frac13\big( \lVert E^{n_0} \rVert^2_{\mathscr{B}}+\frac49\tau^2\lVert E^{n_0} \rVert_{\delta}^2 \big).
\end{array}
\end{equation*}
According to \ref{lem3.2} and Cauchy-Schwarz inequality, it follows from \eqref{eq3.13}-\eqref{eq3.14} that
\begin{equation*}\label{eq3.15}
\begin{array}{ll}
\lVert E^{n_0} \rVert^2_{\mathscr{B}}+\frac49\tau^2\lVert E^{n_0} \rVert_{\delta}^2
\leq
\big( \lVert E^0 \rVert^2_{\mathscr{B}}+\frac{23}{18}\tau^2\lVert E^0 \rVert_{\delta}^2 \big)+\frac{9}{2}\tau L\lVert E^1 \rVert^2_{\mathscr{B}}+\frac{9}{2}\tau L\lVert E^0 \rVert^2_{\mathscr{B}} \\
\quad\qquad\qquad\quad\qquad\qquad+6\tau L\sum\limits^{n_0}_{k=2}\big(  \frac32\lVert E^k \rVert^2_{\mathscr{B}}+\lVert E^{k-1} \rVert^2_{\mathscr{B}}+\frac12\lVert E^{k-2} \rVert^2_{\mathscr{B}}  \big)\\
~\quad\qquad\quad\qquad\qquad\leq
\big( \lVert E^0 \rVert^2+\frac{23}{18}\tau^2\lVert E^0 \rVert_{\delta}^2 \big)+\frac{9}{2}\tau L\lVert E^1 \rVert^2_{\mathscr{B}}+\frac{9}{2}\tau L\lVert E^0 \rVert^2_{\mathscr{B}} \\
~\quad\quad\qquad\quad\qquad\qquad+6\tau L\sum\limits^{n}_{k=2}\big(  \frac32\lVert E^k \rVert^2_{\mathscr{B}}+\lVert E^{k-1} \rVert^2_{\mathscr{B}}+\frac12\lVert E^{k-2} \rVert^2_{\mathscr{B}}  \big)\\
~\quad\qquad\quad\qquad\qquad=\Big(  \big( I+\frac{23}{18}\tau^2\delta^{\alpha}_x\delta^{\beta}_y \big)E^0,E^0  \Big)+\frac{9}{2}\tau L\lVert E^1 \rVert^2_{\mathscr{B}}+\frac{9}{2}\tau L\lVert E^0 \rVert^2_{\mathscr{B}} \\
~\quad\quad\qquad\quad\qquad\qquad+6\tau L\sum\limits^{n}_{k=2}\big(  \frac32\lVert E^k \rVert^2_{\mathscr{B}}+\lVert E^{k-1} \rVert^2_{\mathscr{B}}+\frac12\lVert E^{k-2} \rVert^2_{\mathscr{B}}  \big)\\
~\quad\qquad\quad\qquad\qquad\leq
\sqrt{3}\lVert \big( I+\frac{23}{18}\tau^2\delta^{\alpha}_x\delta^{\beta}_y \big)E^0 \rVert\lVert E^0 \rVert_{\mathscr{B}}  +\frac{9}{2}\tau L\lVert E^1 \rVert^2_{\mathscr{B}}+\frac{9}{2}\tau L\lVert E^0 \rVert^2_{\mathscr{B}} \\
~\quad\qquad\quad\quad\qquad\qquad+6\tau L\sum\limits^{n}_{k=2}\big(  \frac32\lVert E^k \rVert^2_{\mathscr{B}}+\lVert E^{k-1} \rVert^2_{\mathscr{B}}+\frac12\lVert E^{k-2} \rVert^2_{\mathscr{B}}  \big).\\
\end{array}\tag{3.15}
\end{equation*}
We obtain from \eqref{eq3.14} and \ref{lem3.2} that
\begin{equation*}\label{eq3.16}
\begin{array}{ll}
\lVert E^{n_0} \rVert^2_{\mathscr{B}}+\frac49\tau^2\lVert E^{n_0} \rVert_{\delta}^2
\leq
\sqrt{3}\lVert \big( I+\frac{23}{18}\tau^2\delta^{\alpha}_x\delta^{\beta}_y \big)E^0 \rVert\lVert E^0 \rVert_{\mathscr{B}}   +\frac{9}{2}\tau L\lVert E^1 \rVert^2_{\mathscr{B}}+\frac{9}{2}\tau L\lVert E^0 \rVert^2_{\mathscr{B}} \\
\quad\qquad\qquad\quad\qquad\qquad+6\tau L\sum\limits^{n}_{k=2}\big(  \frac32\lVert E^k \rVert^2_{\mathscr{B}}+\lVert E^{k-1} \rVert^2_{\mathscr{B}}+\frac12\lVert E^{k-2} \rVert^2_{\mathscr{B}}  \big)\\
~\quad\qquad\quad\qquad\qquad\leq
\Big( \sqrt{3}\lVert \big( I+\frac{23}{18}\tau^2\delta^{\alpha}_x\delta^{\beta}_y \big)E^0 \rVert   +\frac{9}{2}\tau L\lVert E^1 \rVert_{\mathscr{B}}+\frac{9}{2}\tau L\lVert E^0 \rVert_{\mathscr{B}} \\
\quad\qquad\qquad\quad\qquad\qquad+6\tau L\sum\limits^{n}_{k=2}\big(  \frac32\lVert E^k \rVert_{\mathscr{B}}+\lVert E^{k-1} \rVert_{\mathscr{B}}+\frac12\lVert E^{k-2} \rVert_{\mathscr{B}}  \big)\Big) \sqrt{\lVert E^{n_0} \rVert_{\mathscr{B}}^2+\frac49\tau^2\lVert E^{n_0} \rVert_{\delta}^2}.\\
\end{array}\tag{3.16}
\end{equation*}
Thus, it follows from \eqref{eq3.14} and \eqref{eq3.16} that
\begin{equation*}\label{eq3.17}
\begin{array}{ll}
\lVert E^{n} \rVert_{\mathscr{B}} \leq \sqrt{\lVert E^{n_0} \rVert^2_{\mathscr{B}}+\frac49\tau^2\lVert E^{n_0} \rVert_{\delta}^2}\\
\quad\qquad\leq
\sqrt{3}\lVert \big( I+\frac{23}{18}\tau^2\delta^{\alpha}_x\delta^{\beta}_y \big)E^0 \rVert   +\frac{9}{2}\tau L\lVert E^1 \rVert_{\mathscr{B}}+\frac{9}{2}\tau L\lVert E^0 \rVert_{\mathscr{B}}\\
\quad\qquad\qquad+6\tau L\sum\limits^{n}_{k=2}\big(  \frac32\lVert E^k \rVert_{\mathscr{B}}+\lVert E^{k-1} \rVert_{\mathscr{B}}+\frac12\lVert E^{k-2} \rVert_{\mathscr{B}}  \big) \\
\quad\qquad=
 \sqrt{3}\lVert \big( I+\frac{23}{18}\tau^2\delta^{\alpha}_x\delta^{\beta}_y \big)E^0 \rVert +9\tau L\sum\limits^{n-1}_{k=0}\lVert E^k \rVert_{\mathscr{B}}+6\tau L\sum\limits^{n-1}_{k=0}\lVert E^k \rVert_{\mathscr{B}}+3\tau L\sum\limits^{n-1}_{k=0}\lVert E^k \rVert_{\mathscr{B}}\\
\quad\qquad\qquad+
 9\tau L\lVert E^n \rVert_{\mathscr{B}}-\frac{9}{2}\tau L\lVert E^1 \rVert_{\mathscr{B}}-\frac{21}{2}\tau L\lVert E^0 \rVert_{\mathscr{B}}-3\tau L\lVert E^{n-1} \rVert_{\mathscr{B}}\\
\quad\qquad\leq
\sqrt{3}\lVert \big( I+\frac{23}{18}\tau^2\delta^{\alpha}_x\delta^{\beta}_y \big)E^0 \rVert+18 \tau L\sum\limits^{n-1}_{k=0}\lVert E^k \rVert_{\mathscr{B}}+9\tau L\lVert E^n \rVert_{\mathscr{B}}.
\end{array}\tag{3.17}
\end{equation*}
we can obtain the recursion from \eqref{eq3.17} that
\begin{equation*}\label{eq3.18}
\begin{array}{ll}
(1-9\tau L)\lVert E^n \rVert_{\mathscr{B}} \leq \sqrt{3}\lVert \big( I+\frac{23}{18}\tau^2\delta^{\alpha}_x\delta^{\beta}_y \big)E^0 \rVert+18 \tau L\sum\limits^{n-1}_{k=0}\lVert E^k \rVert_{\mathscr{B}}.
\end{array}\tag{3.18}
\end{equation*}
For any $\nu\in(0,1)$, and $0<\tau\leq \tau_0=\frac{1-\nu}{9L}$, according to \ref{lem3.5}, it follows from \eqref{eq3.18} that
\begin{equation*}\label{eq3.19}
\begin{array}{ll}
\lVert E^n \rVert_{\mathscr{B}} \leq \frac{\sqrt{3}}{\nu}\lVert \big( I+\frac{23}{18}\tau^2\delta^{\alpha}_x\delta^{\beta}_y \big)E^0 \rVert+\frac{18}{\nu} \tau L\sum\limits^{n-1}_{k=0}\lVert E^k \rVert_{\mathscr{B}}\\
\quad\qquad\leq \frac{\sqrt{3}}{\nu}\lVert \big( I+\frac{23}{18}\tau^2\delta^{\alpha}_x\delta^{\beta}_y \big)E^0 \rVert e^{\frac{18}{\nu} \tau Ln}\\
\quad\qquad\leq \frac{\sqrt{3}}{\nu}\lVert \big( I+\frac{23}{18}\tau^2\delta^{\alpha}_x\delta^{\beta}_y \big)E^0 \rVert e^{\frac{18}{\nu}LT}.
\end{array}\tag{3.19}
\end{equation*}
Recalling \ref{lem3.2}, we have
$$
\lVert E^n \rVert \leq \frac{3}{\nu}e^{\frac{18}{\nu}LT}\lVert \big( I+\frac{23}{18}\tau^2\delta^{\alpha}_x\delta^{\beta}_y \big)E^0 \rVert,\qquad n\geq 1.
$$
Therefore, the numerical scheme is stable. The proof is completed.
\end{proof}

\begin{thm}\label{thm3.2}
For any positive real number $v \in (0,1)$, if $0 <\tau \leq \tau_0=\frac{1-\nu}{9L}$, then the numerical scheme \eqref{eq2.13}-\eqref{eq2.15} is convergent, and it holds that $$\underset{1 \leq n \leq N}{\max}\lVert \eta^n \rVert=O(\tau^2+h_x^4+h_y^4).$$
\end{thm}
\begin{proof}
Denote $\eta^n_{i,j}=u(x_i,y_j,t_n)-u^n_{i,j},(x_i,y_j)\in{\Omega}_h, 0\leq n\leq N$, and $\eta^n=(\eta_{0,0}^n,\cdots,\eta_{M_1,0}^n,\cdots,\eta_{0,M_2}^n,\cdots,\eta_{M_1,M_2}^n )$. Similarly with the inference procedure of \ref{thm3.1}, it follows from \eqref{eq3.13} that
\begin{equation*}\label{eq3.20}
\begin{array}{ll}
\lVert \eta^n \rVert^2_{\mathscr{B}}+\frac49\tau^2\lVert \eta^n \rVert_{\delta}^2
\leq
\frac13\big( \lVert \eta^{n-1} \rVert^2_{\mathscr{B}}+\frac49\tau^2\lVert \eta^{n-1} \rVert_{\delta}^2 \big)+\frac23\big( \lVert \eta^0 \rVert^2_{\mathscr{B}}+\frac{23}{18}\tau^2\lVert \eta^0 \rVert_{\delta}^2 \big)\\
\quad\qquad\qquad\quad\qquad\qquad+3\tau L\lVert \eta^1 \rVert^2_{\mathscr{B}}+3\tau L\lVert \eta^0 \rVert^2_{\mathscr{B}}\\
\quad\qquad\qquad\quad\qquad\qquad+4\tau L\sum\limits^{n}_{k=2}\big(  \frac32\lVert \eta^k \rVert^2_{\mathscr{B}}+\lVert \eta^{k-1} \rVert^2_{\mathscr{B}}+\frac12\lVert \eta^{k-2} \rVert^2_{\mathscr{B}}  \big)\\
\quad\qquad\qquad\quad\qquad\qquad+\frac{4\tau h_xh_y}{3}\sum\limits^{n}_{k=2}\sum\limits^{M_1-1}_{i=1}\sum\limits^{M_2-1}_{j=1}\mathscr{R}_{i,j}^k \eta^k_{i,j}+2\tau h_xh_y\sum\limits^{M_1-1}_{i=1}\sum\limits^{M_2-1}_{j=1}\mathscr{R}_{i,j}^1 \eta^1_{i,j}\\
\quad\qquad\quad\qquad\qquad\leq
\frac13\big( \lVert \eta^{n-1} \rVert^2_{\mathscr{B}}+\frac49\tau^2\lVert \eta^{n-1} \rVert_{\delta}^2 \big)+\frac23\big( \lVert \eta^0 \rVert^2_{\mathscr{B}}+\frac{23}{18}\tau^2\lVert \eta^0 \rVert_{\delta}^2 \big)\\
\quad\qquad\qquad\quad\qquad\qquad+3\tau L\lVert \eta^1 \rVert^2_{\mathscr{B}}+3\tau L\lVert \eta^0 \rVert^2_{\mathscr{B}}\\
\quad\qquad\qquad\quad\qquad\qquad+4\tau L\sum\limits^{n}_{k=2}\big(  \frac32\lVert \eta^k \rVert^2_{\mathscr{B}}+\lVert \eta^{k-1} \rVert^2_{\mathscr{B}}+\frac12\lVert \eta^{k-2} \rVert^2_{\mathscr{B}}  \big)\\
\quad\qquad\qquad\quad\qquad\qquad+\frac{4\sqrt{3}\tau}{3}\sum\limits^{n}_{k=2}\lVert \mathscr{R}^k \rVert \lVert \eta^k \rVert_{\mathscr{B}}
+2\sqrt{3}\tau\lVert \mathscr{R}^1 \rVert \lVert \eta^1 \rVert_{\mathscr{B}}.\\
\end{array}\tag{3.20}
\end{equation*}
Similarly with \eqref{eq3.14}-\eqref{eq3.17}, we obtain
\begin{equation*}\label{eq3.21}
\begin{array}{ll}
(1-9\tau L)\lVert \eta^n \rVert_{\mathscr{B}} \leq 2\sqrt{3}\tau\sum\limits^{n}_{k=2}\lVert \mathscr{R}^k \rVert
+3\sqrt{3}\tau\lVert \mathscr{R}^1 \rVert +18 \tau L\sum\limits^{n-1}_{k=0}\lVert \eta^k \rVert_{\mathscr{B}}.
\end{array}\tag{3.21}
\end{equation*}
Assume that $\tau\leq 1$, it follows from \eqref{eq2.12} and the definition of the norm that
\begin{equation*}
\begin{array}{ll}
\tau\lVert \mathscr{R}^1 \rVert\leq c_1\sqrt{(b-a)(d-c)}(\tau^2+\tau h_x^4+\tau h_y^4)\leq c_1\sqrt{(b-a)(d-c)}(\tau^2+ h_x^4+ h_y^4)\\
\quad\qquad=c_1^\ast(\tau^2+ h_x^4+ h_y^4),
\\
\tau\sum\limits^{n-1}_{k=2}\lVert \mathscr{R}^k \rVert\leq \tau(n-2)c_2\sqrt{(b-a)(d-c)}(\tau^2+h_x^4+h_y^4)\leq c_2 T\sqrt{(b-a)(d-c)}(\tau^2+h_x^4+h_y^4)\\
\quad\quad\quad\qquad=c_2^\ast(\tau^2+ h_x^4+ h_y^4),
\end{array}
\end{equation*}
where $c_1^\ast=c_1\sqrt{(b-a)(d-c)}$, $c_2^\ast=c_2 T\sqrt{(b-a)(d-c)}.$\\
For any positive real number $\nu\in(0,1)$, let $\tau_0=\frac{1-\nu}{9L}$, if $0<\tau\leq \tau_0$, it follows from \ref{lem3.5} and \eqref{eq3.21} that
\begin{equation*}\label{eq3.22}
\begin{array}{ll}
\frac{\sqrt{3}}{3}\lVert \eta^n \rVert \leq \lVert \eta^n \rVert_{\mathscr{B}} \leq \frac{2\sqrt{3}\tau}{\nu}\sum\limits^{n}_{k=2}\lVert \mathscr{R}^k \rVert
+\frac{3\sqrt{3}\tau}{\nu}\lVert \mathscr{R}^1 \rVert +\frac{18 \tau L}{\nu}\sum\limits^{n-1}_{k=0}\lVert \eta^k \rVert_{\mathscr{B}}\\
\quad\qquad\leq
\big( \frac{2\sqrt{3}\tau}{\nu}\sum\limits^{n}_{k=2}\lVert \mathscr{R}^k \rVert
+\frac{3\sqrt{3}\tau}{\nu}\lVert \mathscr{R}^1 \rVert \big)e^{\frac{18 n\tau L}{\nu}}\\
\quad\qquad\leq
c_3(\tau^2+h_x^4+h_y^4),
\end{array}\tag{3.22}
\end{equation*}
where $c_3=\frac{\sqrt{3}(2c_1^\ast+3c_2^\ast)}{\nu}e^{\frac{18LT}{\nu}}.$\\
Thus
$$
\underset{1 \leq n \leq N}{\max}\lVert \eta^n \rVert=O(\tau^2+h_x^4+h_y^4).
$$
The proof is completed.
\end{proof}
\section{Numerical experiments}\label{sec4}
Let $\lVert \eta(h,\tau) \rVert=\sqrt{h^2 \sum\limits^{M_1-1}_{i=1} \sum\limits^{M_2-1}_{j=1}\left|  u(x_i,y_j,t_N)-u^N_{i,j}  \right|^2}$ and $\lVert \eta(h,\tau) \rVert_{\infty}=\underset{1\leq j\leq M_2-1}{\underset{1\leq i\leq M_1-1}{\max}}\big| u(x_i,y_j,t_N)-u^N_{i,j} \big|$ denote as $L_2$ norm and maximum norm errors with $h=h_x=h_y$ at $t=t_N$, respectively. The observation orders of $L_2$ norm and maximum norm are defined by
\begin{equation*}
\begin{array}{ll}
Rate_{\tau}=\log_2(\frac{\lVert \eta(h,2\tau) \rVert}{\lVert \eta(h,\tau) \rVert}), \quad Rate_h=\log_2(\frac{\lVert \eta(2h,4\tau) \rVert}{\lVert \eta(h,\tau) \rVert}).
 \end{array}
\end{equation*}
\begin{equation*}
\begin{array}{ll}
Rate_\tau^\infty=\log_2(\frac{\lVert \eta(h,2\tau) \rVert_{\infty}}{\lVert \eta(h,\tau) \rVert_{\infty}}),\quad Rate_h^\infty=\log_2(\frac{\lVert \eta(2h,4\tau) \rVert_{\infty}}{\lVert \eta(h,\tau) \rVert_{\infty}}).
 \end{array}
\end{equation*}

\begin{exa}\label{ex4.1}
Consider the following two-dimensional Riesz space fractional nonlinear reaction-diffusion equation
\begin{equation*}\label{eq4.1a}
\begin{array}{ll}
\frac{\partial u(x,y,t)}{\partial t}={\kappa}_1 \frac{\partial^\alpha u(x,y,t)}{\partial |x|^\alpha}+{\kappa}_2\frac{\partial^\alpha u(x,y,t)}{\partial |y|^\alpha}+g(x,y,t,u(x,y,t)), \qquad 0 < x,y < 1 , 0 < t \leq 1,\tag{4.1$a$}
 \end{array}
\end{equation*}
with boundary and initial conditions
\begin{equation*}\label{eq4.1b}
\begin{array}{ll}
u(0,y,t)=u(x,0,t)=0,\quad u(x,1,t)=u(1,y,t)=0,\quad 0 \leq x,y \leq 1,0 < t \leq 1,
\end{array}\tag{4.1b}
\end{equation*}
\begin{equation*}\label{eq4.1c}
\begin{array}{ll}
u(x,y,0)=x^4(1-x)^4y^4(1-y)^4,\quad 0 \leq x,y \leq 1,
\end{array}\tag{4.1c}
\end{equation*}
where $1 < \alpha,\beta < 2$, the nonlinear source term $g(x,y,t,u(x,y,t))$ is
\begin{equation*}
\begin{array}{rl}
g(x,y,t,u(x,y,t))=\Bigl( u(x,t) \Bigr)^2-e^{-t}y^4(1-y)^4 \Bigl( x^4(1-x)^4+\frac{\Gamma(5)}{\Gamma(5-\alpha)} \left \{ \kappa_1c_{\alpha}x^{4-\alpha}+\kappa_1c_{\alpha}(1-x)^{4-\alpha}\right \}\\
-4\frac{\Gamma(6)}{\Gamma(6-\alpha)} \left \{ \kappa_1c_{\alpha}x^{5-\alpha}+\kappa_1c_{\alpha}(1-x)^{5-\alpha}\right \}+6\frac{\Gamma(7)}{\Gamma(7-\alpha)} \left \{ \kappa_1c_{\alpha}x^{6-\alpha}+\kappa_1c_{\alpha}(1-x)^{6-\alpha}\right \}\\
-4\frac{\Gamma(8)}{\Gamma(8-\alpha)} \left \{ \kappa_1c_{\alpha}x^{7-\alpha}+\kappa_1c_{\alpha}(1-x)^{7-\alpha}\right \}+\frac{\Gamma(9)}{\Gamma(9-\alpha)} \left \{ \kappa_1c_{\alpha}x^{8-\alpha}+\kappa_1c_{\alpha}(1-x)^{8-\alpha}\right \} \Bigr)\\
\qquad\qquad-e^{-t}x^4(1-x)^4 \Bigl( e^{-t}x^4(1-x)^4y^8(1-y)^8+ \frac{\Gamma(5)}{\Gamma(5-\beta)} \left \{ \kappa_2c_{\beta}y^{4-\beta}+\kappa_2c_{\beta}(1-y)^{4-\beta}\right \}\\
-4\frac{\Gamma(6)}{\Gamma(6-\beta)} \left \{ \kappa_2c_{\beta}y^{5-\beta}+\kappa_2c_{\beta}(1-y)^{5-\beta}\right \}+6\frac{\Gamma(7)}{\Gamma(7-\beta)} \left \{ \kappa_2c_{\beta}y^{6-\beta}+\kappa_2c_{\beta}(1-y)^{6-\beta}\right \}\\
-4\frac{\Gamma(8)}{\Gamma(8-\beta)} \left \{ \kappa_2c_{\beta}y^{7-\beta}+\kappa_2c_{\beta}(1-y)^{7-\beta}\right \}+\frac{\Gamma(9)}{\Gamma(9-\beta)} \left \{ \kappa_2c_{\beta}y^{8-\beta}+\kappa_2c_{\beta}(1-y)^{8-\beta}\right \} \Bigr).
\end{array}
\end{equation*}
The exact solution of the problem $\eqref{eq4.1a}-\eqref{eq4.1c}$ is $$u(x,y,t)=e^{-t}x^4(1-x)^4y^4(1-y)^4.$$
\end{exa}

\renewcommand\thefigure{\thesection.\arabic{figure}}
\renewcommand{\thetable}{\thesection.\arabic{table}}
\begin{table}[!htbp]
\centering
\begin{minipage}[t]{01\textwidth}
\centering
\vspace{-0.8em}
\setlength{\abovecaptionskip}{0pt}
\setlength{\belowcaptionskip}{10pt}
\caption{ Errors and corresponding spatial observation orders of BCIM for $\kappa_1=2$, $\kappa_2=4$.}\label{lab4.1}
%\begin{center}
%{  Table~4.1~~The maximum norm and $L_2$ norm errors at t=1 and corresponding spatial observation orders} \\
%\rowcolors{4}{gray!40}{gray!40}
\begin{tabular}{*{8}{c}}
\bottomrule
%\multirow{2}*{$\alpha$} &\multirow{2}*{$\beta$} & \multirow{2}*{$h$} & \multirow{2}*{$\tau$} & \multicolumn{4}{c}{BCIM}  \\
%%\cmidrule(ll){3-5}\cmidrule(ll){6-8}
%\cmidrule(ll){5-8}

{$\alpha$} & {$\beta$} & {$h$} & {$\tau$} & $\lVert \eta(h,\tau) \rVert_{\infty}$ & $Rate_h^\infty$ & $\lVert \eta(h,\tau) \rVert$ & $Rate_h$  \\
\midrule
%%\midrule
$1.1$ & $1.5$ & $\frac{1}{8}$ & $\frac{1}{64}$  & 2.3306e-08 & * & 9.3070e-09 & *  \\
%%\midrule
%%\rowcolor[gray]{0.8}
& & $\frac{1}{16}$ & $\frac{1}{256}$  & 1.3729e-09 & 4.085 & 5.4995e-10 & 4.081   \\
%%\midrule
 & & $\frac{1}{32}$ & $\frac{1}{1024}$  &  8.1699e-11 & 4.071 & 3.1847e-11 &   4.110    \\
%%\midrule
& & $\frac{1}{64}$ & $\frac{1}{4096}$  &  4.7570e-12 &  4.102 & 1.8390e-12 &  4.114    \\
%%\midrule
\midrule
\midrule
%%\midrule
$1.3$ & $1.7$ & $\frac{1}{8}$ & $\frac{1}{64}$  & 2.8704e-08 & * & 1.1769e-08 & *  \\
%%\midrule
%%\rowcolor[gray]{0.8}
& & $\frac{1}{16}$ & $\frac{1}{256}$  & 1.7359e-9 &  4.048 & 7.1648e-10 & 4.038   \\
%%\midrule
 & & $\frac{1}{32}$ & $\frac{1}{1024}$  & 1.0488e-10 &  4.049 &  4.2172e-11 & 4.087    \\
%%\midrule
& & $\frac{1}{64}$ & $\frac{1}{4096}$  & 5.9342e-12 &  4.144 &  2.3909e-12 &  4.141    \\
%%\midrule
\midrule
\midrule
%%\midrule
$1.5$ & $1.9$ & $\frac{1}{8}$ & $\frac{1}{64}$  & 3.4947e-08 & * & 1.4841e-08 & *  \\
%%\midrule
%%\rowcolor[gray]{0.8}
& & $\frac{1}{16}$ & $\frac{1}{256}$  & 2.1909e-09 & 3.996 & 9.4251e-10 &  3.977   \\
%%\midrule
 & & $\frac{1}{32}$ & $\frac{1}{1024}$  &  1.3642e-10 &  4.005 & 5.7485e-11 &  4.035    \\
%%\midrule
& & $\frac{1}{64}$ & $\frac{1}{4096}$  &  8.1762e-12 &  4.061 &  3.3763e-12 & 4.090\label{tab4.1}    \\
%%\midrule
\midrule
\midrule
%%\midrule
$1.8$ & $1.8$ & $\frac{1}{8}$ & $\frac{1}{64}$  &   3.0917e-08 & * & 1.5445e-08 & *  \\
%%\midrule
%%\rowcolor[gray]{0.8}
& & $\frac{1}{16}$ & $\frac{1}{256}$  & 1.9047e-09 &  4.021 &  9.6607e-10 & 3.999   \\
%%\midrule
 & & $\frac{1}{32}$ & $\frac{1}{1024}$  & 1.1627e-10 & 4.034 & 5.8275e-11 & 4.051    \\
%%\midrule
& & $\frac{1}{64}$ & $\frac{1}{4096}$  & 6.8754e-12 &  4.080 &  3.4108e-12 &  4.095    \\
\bottomrule
\end{tabular}
\end{minipage}
\end{table}
\begin{table}[!htbp]
\centering
\begin{minipage}[t]{01\textwidth}
\centering
\vspace{-2em}
\setlength{\abovecaptionskip}{0pt}
\setlength{\belowcaptionskip}{10pt}
\caption{Errors and corresponding temporal observation orders of BCIM for $\kappa_1=2$, $\kappa_2=4$.}\label{tab4.2}
%\begin{center}
%{ {\it Table~4.2~~The maximum norm and $L_2$ norm errors at t=1 and corresponding temporal observation orders}\\
\begin{tabular}{*{8}{c}}
\bottomrule
%\multirow{2}*{$\alpha$} &\multirow{2}*{$\beta$} & \multirow{2}*{$h$} & \multirow{2}*{$\tau$} & \multicolumn{4}{c}{BCIM}  \\
%%\cmidrule(ll){3-5}\cmidrule(ll){6-8}
%\cmidrule(ll){5-8}

{$\alpha$} & {$\beta$} & {$h$} & {$\tau$} & $\lVert \eta(h,\tau) \rVert_{\infty}$ & $Rate_\tau^\infty$ & $\lVert \eta(h,\tau) \rVert$ & $Rate_\tau$ \\
\midrule
%%\midrule
$1.1$ & $1.5$ & $\frac{1}{200}$ & $\frac{1}{10}$  & 1.8993e-07 & * & 4.9124e-08 & *  \\
%%\midrule
%%\rowcolor[gray]{0.8}
& & $\frac{1}{200}$ & $\frac{1}{20}$  &  4.5950e-08 &  2.047 & 1.1901e-08 & 2.045   \\
%%\midrule
 & & $\frac{1}{200}$ & $\frac{1}{40}$  &   1.1396e-08 &  2.012 &  2.9524e-09 &  2.011   \\
%%\midrule
& & $\frac{1}{200}$ & $\frac{1}{80}$  &  2.8441e-09 &  2.002 & 7.3692e-10 &  2.002   \\
%%\midrule
\midrule
\midrule
%%\midrule
$1.3$ & $1.7$ & $\frac{1}{200}$ & $\frac{1}{10}$  & 2.6925e-07 & * & 6.9248e-08 & *  \\
%%\midrule
%%\rowcolor[gray]{0.8}
& & $\frac{1}{200}$ & $\frac{1}{20}$  &  6.3829e-08 &   2.077 & 1.6430e-08 &  2.075   \\
%%\midrule
 & & $\frac{1}{200}$ & $\frac{1}{40}$  &   1.5759e-08 &  2.018 & 4.0577e-09 &  2.018   \\
%%\midrule
& & $\frac{1}{200}$ & $\frac{1}{80}$  &  3.9284e-09 &  2.004 &  1.0115e-09 &  2.004   \\
%%\midrule
\midrule
\midrule
%%\midrule
$1.5$ & $1.9$ & $\frac{1}{200}$ & $\frac{1}{10}$  &  3.9090e-07 & * & 1.0050e-07 & *  \\
%%\midrule
%%\rowcolor[gray]{0.8}
& & $\frac{1}{200}$ & $\frac{1}{20}$  &  8.9948e-08 &  2.120 & 2.3066e-08 & 2.123   \\
%%\midrule
 & & $\frac{1}{200}$ & $\frac{1}{40}$  &  2.2049e-08 &  2.028 &  5.6540e-09 &  2.028   \\
%%\midrule
& & $\frac{1}{200}$ & $\frac{1}{80}$  &  5.4863e-09 &  2.007 & 1.4069e-09 & 2.007   \\
\midrule
\midrule
%%\midrule
$1.8$ & $1.8$ & $\frac{1}{200}$ & $\frac{1}{10}$  & 5.3208e-07 & * & 1.3671e-07 & *  \\
%%\midrule
%%\rowcolor[gray]{0.8}
& & $\frac{1}{200}$ & $\frac{1}{20}$  &  1.2043e-07 &   2.143 &  3.0554e-08 &  2.162   \\
%%\midrule
 & & $\frac{1}{200}$ & $\frac{1}{40}$  &   2.9191e-08 &  2.045 &  7.3857e-09 &  2.049   \\
%%\midrule
& & $\frac{1}{200}$ & $\frac{1}{80}$  &  7.2445e-09 &   2.011 & 1.8319e-09 &  2.011   \\
%%\midrule
\bottomrule
\end{tabular}
\end{minipage}
\end{table}
\par
We use the method \eqref{eq2.17}-\eqref{eq2.19} (abbr. BCIM) to solve \ref{ex4.1} with several values of $h$, $\tau$ and $\alpha$, $\beta$, respectively, the numerical results are listed in \ref{tab4.1}-\ref{tab4.2}. From the results, we can affirm that the fourth order in spatial direction and the second order in temporal direction are in consistent with our theoretical analysis.
\par
For contrast, we also apply the methods (abbr. ADIM and CDIM) in \cite{space_time_one_order} and \cite{space_time_one_two_order} to solve \ref{ex4.1}, respectively. The numerical results are listed in \ref{tab4.3}-\ref{tab4.5}. Numerical results show that BCIM has the more accurate solutions than ADIM and CDIM with the same conditions. It is obvious to find from \ref{tab4.5} that the three schemes generate the same accuracy for the same temporal grid-size, while the BCIM scheme needs fewer spatial grid points and less CPU time than CDIM and ADIM. This means that BCIM scheme reduces storage requirement and CPU time successfully. All the computations were carried out using MATLAB R2014a software on a HP 288 Pro G2 MT computer, Intel(R) Core(TM) i5-6500, 3.2 GHz CPU machine and 8 GB RAM.
\begin{table}[!htbp]
\centering
\begin{minipage}[t]{01\textwidth}
\centering
\vspace{-1em}
\setlength{\abovecaptionskip}{0pt}
\setlength{\belowcaptionskip}{10pt}
\caption{Errors of numerical methods for $\kappa_1=\kappa_2=0.5$. }\label{tab4.3}
%{ {\it Table~4.3~~The maximum norm and $L_2$ norm errors at t=1 }  \\
\begin{tabular}{*{8}{c}}
\bottomrule
\multirow{2}*{$\alpha$} &\multirow{2}*{$\beta$} & \multirow{2}*{$h$} & \multirow{2}*{$\tau$} & \multicolumn{2}{c}{BCIM} & \multicolumn{2}{c}{ADIM\cite{space_time_one_order}} \\
%%\cmidrule(ll){3-5}\cmidrule(ll){6-8}
\cmidrule(ll){5-6}
\cmidrule(ll){7-8}
&&&& $\lVert \eta(h,\tau) \rVert_{\infty}$  & $\lVert \eta(h,\tau) \rVert$ & $\lVert \eta(h,\tau) \rVert_{\infty}$  & $\lVert \eta(h,\tau) \rVert$ \\
\midrule
%%\midrule
$1.1$ & $1.1$ & $\frac{1}{40}$ & $\frac{1}{40}$  &  2.1747e-09 & 5.7548e-10 & 1.7085e-06 & 4.4962e-07  \\
%%\midrule
%%\rowcolor[gray]{0.8}
& & $\frac{1}{80}$ & $\frac{1}{80}$  &  5.4848e-10 & 1.4583e-10 & 1.0019e-06 & 2.5965e-07  \\
%%\midrule
 & & $\frac{1}{160}$ & $\frac{1}{160}$  &  1.4706e-10 & 3.9499e-11 & 5.5006e-07 & 1.4147e-07  \\
%%\midrule
& & $\frac{1}{320}$ & $\frac{1}{320}$  &  4.7129e-11 & 1.3134e-11 & 2.8951e-07 &  7.4208e-08 \\
\midrule
\midrule
$1.5$ & $1.5$ & $\frac{1}{40}$ & $\frac{1}{40}$  &  3.5315e-09 & 9.0144e-10 & 4.6111e-07 & 1.1639e-07  \\
%%\midrule
%%\rowcolor[gray]{0.8}
& & $\frac{1}{80}$ & $\frac{1}{80}$  &  8.8246e-10 & 2.2586e-10 & 2.4185e-07 & 6.0717e-08  \\
%%\midrule
 & & $\frac{1}{160}$ & $\frac{1}{160}$  &  2.2544e-10 & 5.7922e-11 & 1.2391e-07 & 3.1030e-08  \\
%%\midrule
& & $\frac{1}{320}$ & $\frac{1}{320}$  &  6.1618e-11 & 1.6083e-11 & 6.2718e-08 & 1.5688e-08  \\
\midrule
\midrule
$1.9$ & $1.9$ & $\frac{1}{40}$ & $\frac{1}{40}$  &  6.4489e-09 & 1.6269e-09 & 4.7529e-07 & 1.1933e-07  \\
%%\midrule
%%\rowcolor[gray]{0.8}
& & $\frac{1}{80}$ & $\frac{1}{80}$  &  1.6102e-09 & 4.0648e-10 & 2.5926e-07 & 6.5003e-08  \\
%%\midrule
 & & $\frac{1}{160}$ & $\frac{1}{160}$  &  4.0480e-10 & 1.0224e-10 & 1.3611e-07 & 3.4169e-08  \\
%%\midrule
& & $\frac{1}{320}$ & $\frac{1}{320}$  &  1.0376e-10 &  2.6251e-11 & 6.9853e-08 & 1.7559e-08  \\
%%\midrule
\bottomrule
%\end{tabular}
%}
%\end{center}
\end{tabular}
\end{minipage}
\end{table}

\begin{table}[!htbp]
\centering
\begin{minipage}[t]{01\textwidth}
\centering
\vspace{-1em}
\setlength{\abovecaptionskip}{0pt}
\setlength{\belowcaptionskip}{10pt}
\caption{Errors and corresponding observation orders of numerical methods for $\kappa_1=\kappa_2=1.5$. }\label{tab4.4}
%{ {\it Table~4.3~~The maximum norm and $L_2$ norm errors at t=1 }  \\
\begin{tabular}{*{8}{c}}
\bottomrule
\multirow{2}*{$\alpha$} &\multirow{2}*{$\beta$} & \multirow{2}*{$h$} & \multirow{2}*{$\tau$} & \multicolumn{2}{c}{BCIM} & \multicolumn{2}{c}{CDIM\cite{space_time_one_two_order}} \\
%%\cmidrule(ll){3-5}\cmidrule(ll){6-8}
\cmidrule(ll){5-6}
\cmidrule(ll){7-8}
&&&& $\lVert \eta(h,\tau) \rVert$  & $Rate_h$ & $\lVert \eta(h,\tau) \rVert$  &  $Rate_h$\\
\midrule
%%\midrule
$1.1$ & $1.1$ & $\frac{1}{5}$ & $\frac{1}{25}$  &  4.1999e-08 & * & 1.0670e-07 & *  \\
%%\midrule
%%\rowcolor[gray]{0.8}
& & $\frac{1}{10}$ & $\frac{1}{100}$  &  2.6112e-09 & 4.008 & 2.7498e-08 & 1.956  \\
%%\midrule
 & & $\frac{1}{20}$ & $\frac{1}{400}$  &  1.5119e-10 & 4.110 & 7.0472e-09 & 1.964  \\
%%\midrule
& & $\frac{1}{40}$ & $\frac{1}{1600}$  &  9.1834e-12 & 4.041 & 1.7749e-09 & 1.989  \\
%%\midrule
\midrule
\midrule
$1.5$ & $1.5$ & $\frac{1}{5}$ & $\frac{1}{25}$  &  6.6479e-08 & * & 1.8701e-07 & *  \\
& & $\frac{1}{10}$ & $\frac{1}{100}$  &   4.3835e-09 & 3.923 & 5.3111e-08 & 1.816  \\
%%\midrule
 & & $\frac{1}{20}$ & $\frac{1}{400}$  &  2.5617e-10 & 4.097 & 1.4296e-08 & 1.893  \\
%%\midrule
& & $\frac{1}{40}$ & $\frac{1}{1600}$  &  1.4680e-11 & 4.125 & 3.6608e-09 & 1.965  \\
%%\midrule
\midrule
\midrule
$1.9$ & $1.9$ & $\frac{1}{5}$ & $\frac{1}{25}$  &  9.8521e-08 & * & 3.2851e-07 & * \\
& & $\frac{1}{10}$ & $\frac{1}{100}$  &  7.1405e-09 & 3.786 & 1.0507e-07 & 1.645  \\
%%\midrule
 & & $\frac{1}{20}$ & $\frac{1}{400}$  &  4.4896e-10 & 3.991 & 3.0700e-08 & 1.775  \\
%%\midrule
& & $\frac{1}{40}$ & $\frac{1}{1600}$  &  2.7350e-11 & 4.037 & 8.1275e-09 & 1.917  \\
%%\midrule
\bottomrule
%\end{tabular}
%}
%\end{center}
\end{tabular}
\end{minipage}
\end{table}

\begin{table}[!htbp]
\centering
\begin{minipage}[t]{01\textwidth}
\centering
\vspace{-1em}
\setlength{\abovecaptionskip}{0pt}
\setlength{\belowcaptionskip}{10pt}
\caption{Errors and corresponding CPU time costs of numerical methods for $\kappa_1=\kappa_2=0.5$. }\label{tab4.5}
%{  { \it Table~4.4~~The maximum norm errors at t=1 and theirs corresponding $CPU$ time costs} \\
\begin{tabular}{*{12}{c}}
\bottomrule
\multirow{2}*{$\alpha$} &\multirow{2}*{$\beta$} & \multirow{2}*{$\tau$} & \multicolumn{3}{c}{BCIM} & \multicolumn{3}{c}{CDIM\cite{space_time_one_two_order}} & \multicolumn{3}{c}{ADIM\cite{space_time_one_order}}  \\
%%\cmidrule(ll){3-5}\cmidrule(ll){6-8}
\cmidrule(ll){4-6}
\cmidrule(ll){7-9}
\cmidrule(ll){10-12}
&&& $h$ & $\lVert \eta(h,\tau) \rVert_{\infty}$  &CPU(s) & $h$ & $\lVert \eta(h,\tau) \rVert_{\infty}$  &CPU(s) & $h$ & $\lVert \eta(h,\tau) \rVert_{\infty}$  &CPU(s) \\
\midrule
%%\midrule
$1.1$ & $1.1$ & $\frac{1}{144}$ & $\frac{1}{12}$  & 2.4853e-09 & 1.480& $\frac{1}{100}$ & 4.4803e-08 & 9.506 & $\frac{1}{144}$ & 6.0451e-07 &16.702  \\
%%\midrule
%%\rowcolor[gray]{0.8}
& & $\frac{1}{196}$ & $\frac{1}{14}$  &  1.3567e-09 &  2.372 & $\frac{1}{100}$ & 3.2719e-08 & 12.839 & $\frac{1}{196}$ & 4.5739e-07 & 45.000   \\
%%\midrule
 &&  $\frac{1}{256}$ & $\frac{1}{16}$  &  8.0620e-10 &  3.665 & $\frac{1}{100}$ &  2.4841e-08 &  16.607 & $\frac{1}{256}$ & 3.5716e-07 & 103.324  \\
%%\midrule
& & $\frac{1}{400}$ & $\frac{1}{20}$  &  3.4227e-10 &  7.418 & $\frac{1}{100}$ &  1.5540e-08 &  25.941 & $\frac{1}{400}$ & 2.3410e-07 & 477.229 \\

& & $\frac{1}{576}$ & $\frac{1}{24}$  & 1.7382e-10 & 13.295& $\frac{1}{100}$ & 1.0471e-08 & 40.944 & $\frac{1}{576}$ & 1.6473e-07 & 2089.920  \\

& & $\frac{1}{784}$ & $\frac{1}{28}$  & 1.0093e-10 & 21.833& $\frac{1}{100}$ & 7.4096e-09 & 48.423 & $\frac{1}{784}$ & 1.2201e-07 & 7282.158  \\
\midrule
\midrule
%%\rowcolor[gray]{0.8}
$1.5$ & $1.5$ & $\frac{1}{144}$ & $\frac{1}{12}$  & 3.5317e-09 & 1.452& $\frac{1}{100}$ & 7.9016e-08 & 8.295 & $\frac{1}{144}$ & 1.3730e-07 & 18.782 \\
%%\midrule
%%\rowcolor[gray]{0.8}
& & $\frac{1}{196}$ & $\frac{1}{14}$  &  1.9340e-09 &  2.367 & $\frac{1}{100}$ & 5.7999e-08 & 11.372 & $\frac{1}{196}$ & 1.0161e-07 & 47.309   \\
%%\midrule
 &&  $\frac{1}{256}$ & $\frac{1}{16}$  &  1.0980e-09 &  3.686 & $\frac{1}{100}$ &  4.4223e-08 &  14.788 & $\frac{1}{256}$ & 7.8158e-08 & 107.309 \\
%%\midrule
& & $\frac{1}{400}$ & $\frac{1}{20}$  &  4.4652e-10 &  7.369 & $\frac{1}{100}$ &  2.7885e-08 &  23.498 & $\frac{1}{400}$ & 5.0297e-08 & 465.193 \\

& & $\frac{1}{576}$ & $\frac{1}{24}$  & 2.0591e-10 & 13.497& $\frac{1}{100}$ & 1.8946e-08 & 33.295 & $\frac{1}{576}$ & 3.5032e-08 & 2020.553  \\

& & $\frac{1}{784}$ & $\frac{1}{28}$  & 1.1094e-10 & 22.479& $\frac{1}{100}$ & 1.3534e-08 & 45.220 & $\frac{1}{784}$ & 2.5783e-08 & 7338.858  \\
\midrule
\midrule
$1.9$ & $1.9$ & $\frac{1}{144}$ & $\frac{1}{12}$  & 5.7978e-09 & 1.450& $\frac{1}{100}$ & 1.4748e-07 & 8.290 & $\frac{1}{144}$ & 1.5039e-07 & 19.943  \\
%%\midrule
%%\rowcolor[gray]{0.8}
& & $\frac{1}{196}$ & $\frac{1}{14}$  &  3.2021e-09 &  2.364 & $\frac{1}{100}$ & 1.0932e-07 & 11.539 & $\frac{1}{196}$ & 1.1216e-07 & 47.362   \\
%%\midrule
 &&  $\frac{1}{256}$ & $\frac{1}{16}$  &  1.8549e-09 &  3.619 & $\frac{1}{100}$ &  8.4002e-08 &  14.873 & $\frac{1}{256}$ & 8.6739e-08 & 108.434  \\
%%\midrule
& & $\frac{1}{400}$ & $\frac{1}{20}$  &  7.6365e-10 &  7.375 & $\frac{1}{100}$ &  5.3645e-08 &  23.170 & $\frac{1}{400}$ & 5.6183e-08 & 477.779 \\

& & $\frac{1}{576}$ & $\frac{1}{24}$  & 3.6296e-10 & 13.281& $\frac{1}{100}$ & 3.6878e-08 & 33.192 & $\frac{1}{576}$ & 3.9275e-08 & 2074.364  \\

& & $\frac{1}{784}$ & $\frac{1}{28}$  & 1.9712e-10 & 21.817& $\frac{1}{100}$ & 2.6671e-08 & 45.235 & $\frac{1}{784}$ &  2.8971e-08 & 7300.367  \\
\bottomrule
%\end{tabular}
%}
%\end{center}
\end{tabular}
\end{minipage}
\end{table}

\begin{exa}
Consider the following fractional FitzHugh-Nagumo model\cite{Bueno-Orovio}
\begin{equation*}\label{eq4.2a}
\left \{
\begin{array}{ll}
    \frac{\partial u}{\partial t}=\kappa_1\frac{\partial^{\alpha}u}{\partial|x|^{\alpha}}+\kappa_2\frac{\partial^{\beta}u}{\partial|y|^{\beta}}+u(1-u)(u-\mu)-w,\quad
    (x,y,t)\in (0,2.5)\times (0,2.5)\times(0,T],\\
\frac{\partial w}{\partial t}=\varepsilon(\lambda u -\gamma w-\delta),\quad (x,y,t)\in (0,2.5)\times (0,2.5)\times(0,T],
\end{array}
\right.\tag{4.2$a$}
\end{equation*}
where $\delta = 0, \mu = 0.1,\gamma=1,\varepsilon=0.01,\lambda=0.5$, and $\kappa_1,\kappa_2$ are nonnegative diffusion coefficients.\\
The initial-boundary conditions of \eqref{eq4.2a} are taken as
\begin{equation*}\label{eq4.2b}
\begin{array}{ll}
u(x,y,0)=\left \{
\begin{array}{ll}
    1,\quad (x,y)\in (0,1.25]\times (0,1.25),\\
    0,\quad elsewhere,\tag{4.2b}
\end{array}
\right.
\\\\
w(x,y,0)=\left \{
\begin{array}{ll}
    0,\quad (x,y)\in (0,2.5)\times (0,1.25),\\
    0.1,\quad (x,y)\in (0,2.5)\times [1.25,2.5),
\end{array}
\right.
\end{array}
\end{equation*}
and
\begin{equation*}\label{eq4.2c}
\begin{array}{ll}
u(0,y,t)=u(2.5,y,t)=u(x,0,t)=u(x,2.5,t)=0,\\
w(0,y,t)=w(2.5,y,t)=w(x,0,t)=w(x,2.5,t)=0,
\end{array}
\begin{array}{ll}
0 \leq x,y \leq 2.5,\quad 0 < t\leq T.
\end{array}
\end{equation*}

\end{exa}
Furthermore, BCIM can also be extended to solve the fractional FitzHugh-Nagumo model which is applied for depicting the propagation of the electrical potential in heterogeneous cardiac tissue. In the simulation, we set the parameters $M_1=M_2=200,N=2000,T=1000$, the results of the simulation at $t=1000$ are shown in \ref{fig4.1} and \ref{fig4.2}. We observed that the numerical solution of fractional FitzHugh-Nagumo model is related to the fractional orders $\alpha$ and $\beta$. The more details of fractional FitzHugh-Nagumo model can refer to \cite{Bueno-Orovio}.

\begin{figure}[!htbp]
\centering
\begin{minipage}[t]{01\textwidth}
\centering
\includegraphics[width=12cm,height=5.5cm]{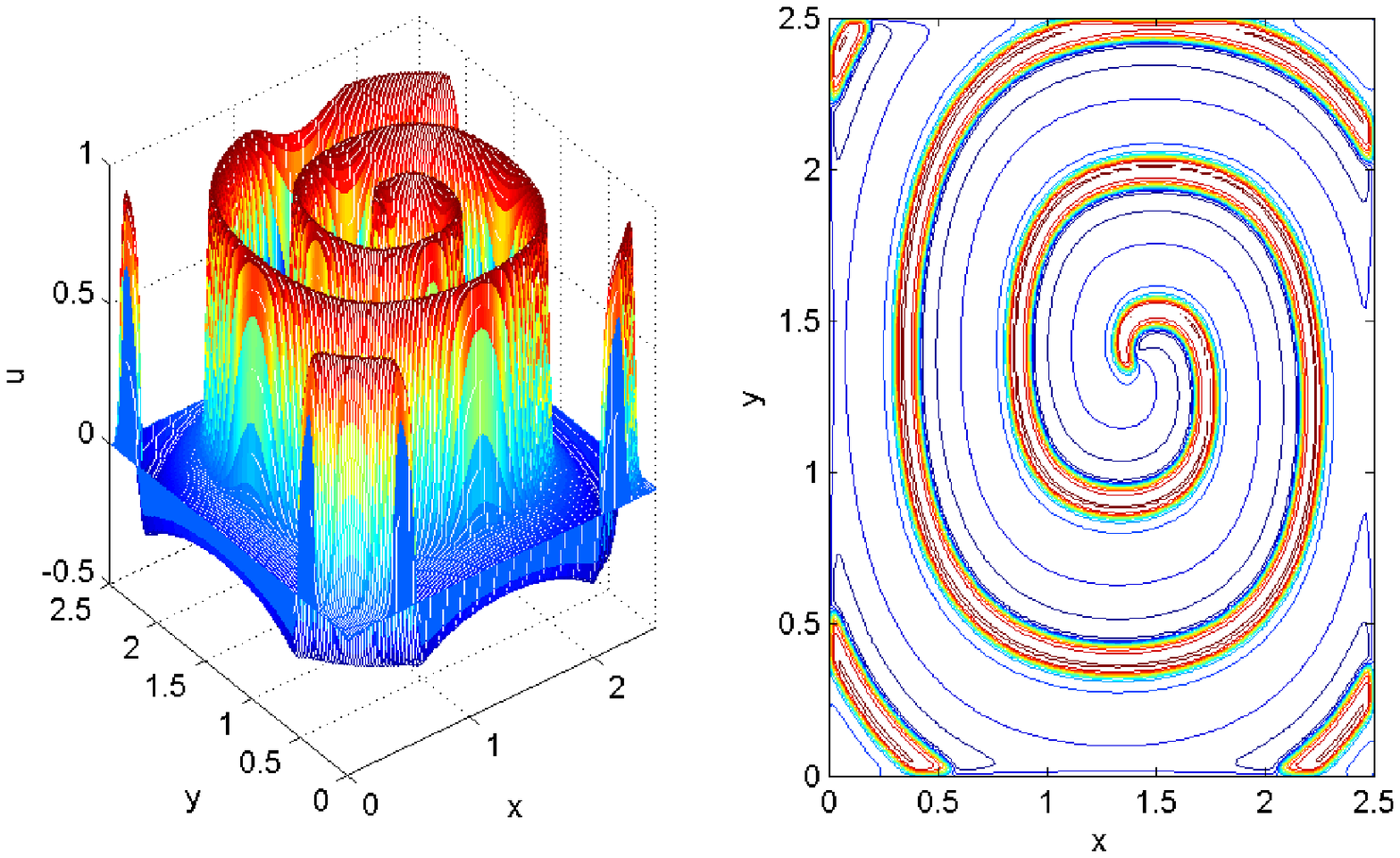}
\caption{\scriptsize ~~Numerical solution of the FitzHugh-Nagumo model with $\kappa_1=\kappa_2=1e-4,\alpha=1.7,\beta=1.7$ at $t=1000$.}\label{fig4.1}
\end{minipage}
\centering
\begin{minipage}[t]{01\textwidth}
\centering
\includegraphics[width=12cm,height=5.5cm]{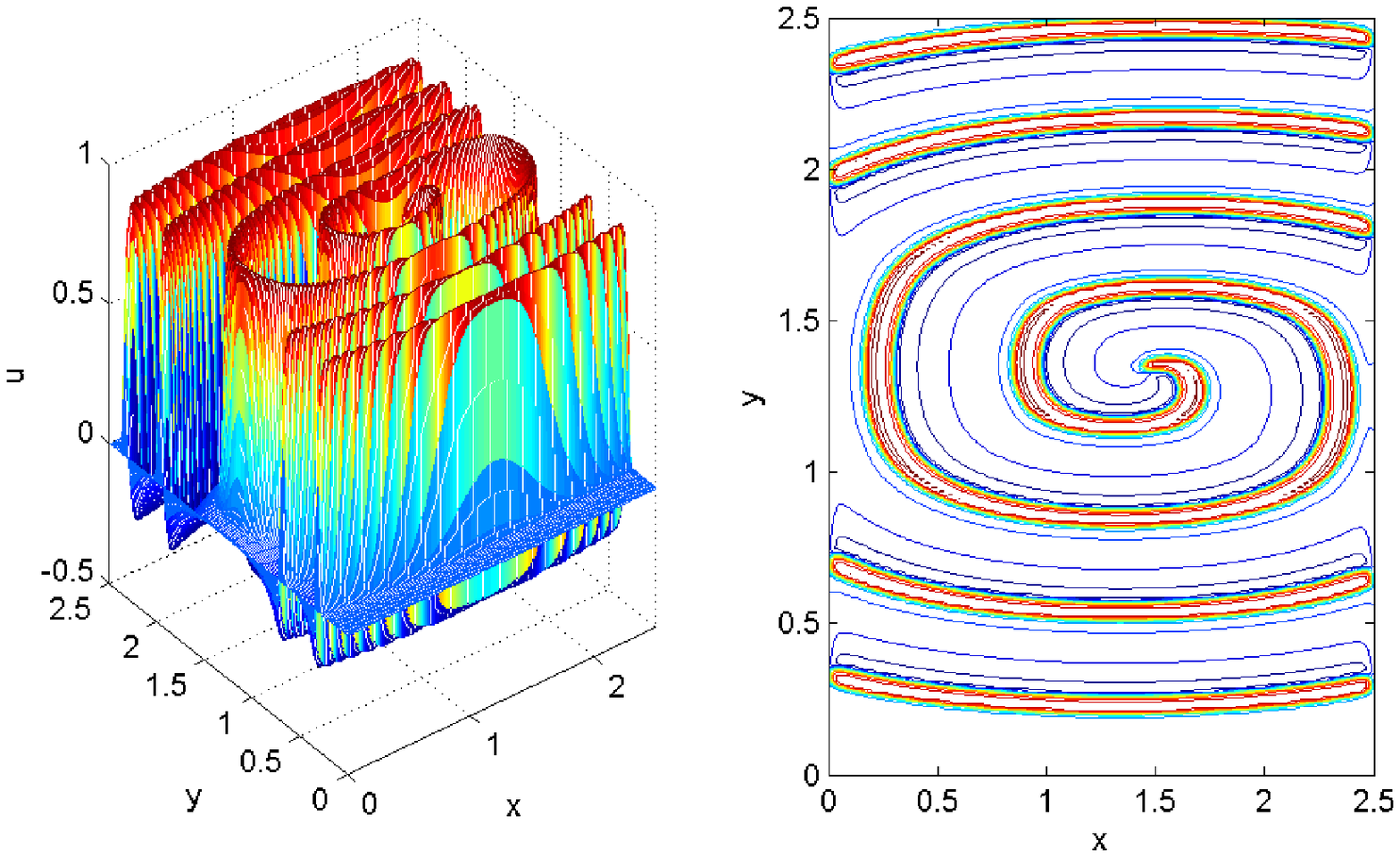}
\caption{\scriptsize ~~Numerical solution of the FitzHugh-Nagumo model with $\kappa_1=\kappa_2=1e-4,\alpha=1.8,\beta=1.5$ at $t=1000$.}\label{fig4.2}
\end{minipage}

\end{figure}

\bigskip

\section{Conclusion}
In this paper, a compact ADI finite difference scheme is constructed for the two dimensional Riesz space fractional nonlinear reaction-diffusion equation based on the linearized approximations for nonlinear source term. It's proved that the proposed method is stable and convergent with second-order temporal accuracy and fourth-order spatial accuracy by energy method. Finally, the numerical tests verified the correctness of the theoretical analysis and effectiveness of the proposed scheme by comparing with numerical schemes in \cite{space_time_one_two_order,space_time_one_order}.

\section*{Acknowledgements}
 The authors would like to express the thanks to the referees for their valuable comments and suggestions.

\section*{Disclosure statement}
No potential conflict of interest was reported by the authors.

\section*{Funding}
This work is supported by National Science Foundation of China (No. 11671343), and Project of Scientific Research Fund of Hunan Provincial Science and Technology Department (No. 2018WK4006).


%\def\ack{\section*{Acknowledgements}%
 % \addtocontents{toc}{\protect\vspace{6pt}}%
  %\addcontentsline{toc}{section}{Acknowledgements}%
%}
%\ack{The work was supported by the Natural Science Foundation of China (Grant No.11271311,No11371302)




\end{document}